\documentclass[11pt]{article}
\usepackage[utf8]{inputenc}
\usepackage[T1]{fontenc}
\usepackage{comment}
\usepackage{authblk}
\usepackage{esvect}
\usepackage{relsize}
\usepackage{changepage}
\usepackage{mathtools}
\usepackage[english]{babel}
\usepackage{amsmath}
\usepackage{amsthm}
\usepackage{amsfonts}
\usepackage{amssymb}
\usepackage{graphicx}
\usepackage[usenames,dvipsnames]{color}
\theoremstyle{plain}
\newtheorem{theorem}{Theorem}[section]
\newtheorem{corollary}[theorem]{Corollary}
\newtheorem{claim}[theorem]{Claim}

\newtheorem{openproblem}[theorem]{Open problem}
\newtheorem{lemma}[theorem]{Lemma}
\newtheorem{sublemma}[theorem]{Claim}
\newtheorem{observation}[theorem]{Observation}

\makeatletter
\newcommand{\vast}{\bBigg@{4}}
\newcommand{\Vast}{\bBigg@{5}}
\makeatother
\definecolor{bulgarianrose}{rgb}{0.28, 0.02, 0.03}
\definecolor{gray}{rgb}{0.5, 0.5, 0.5}

\theoremstyle{definition}

\theoremstyle{remark}

\newtheorem*{fact*}{Fact}
\newtheorem*{question*}{Question}
\usepackage[left=2cm,right=2cm,top=2cm,bottom=2cm]{geometry}
\usepackage{amssymb}
\usepackage{hyperref}
\makeatletter
\def\namedlabel#1#2{\begingroup
    #2%
    \def\@currentlabel{#2}%
    \phantomsection\label{#1}\endgroup
}
\usepackage[normalem]{ulem}
\usepackage{dsfont}
\usepackage{accents}
\usepackage{verbatim}
\usepackage{ulem}
\usepackage{pgf,tikz,pgfplots}
\pgfplotsset{compat=1.16}
\usepackage[all]{xy} 

\usepackage{stackengine}
\stackMath
\newcommand\tsup[2][2]{%
 \def\useanchorwidth{T}%
  \ifnum#1>1%
    \stackon[-.5pt]{\tsup[\numexpr#1-1\relax]{#2}}{\scriptscriptstyle\sim}%
  \else%
    \stackon[.5pt]{#2}{\scriptscriptstyle\sim}%
  \fi%
}

\newcommand{\R}{{\mathbb R}}
\newcommand{\N}{{\mathbb N}}
\newcommand{\RG} {\ensuremath{\mathcal G(n,r)}}
\newcommand{\RGp} {\ensuremath{\mathcal G_{Po}(n,r)}}
\newcommand{\RT} {\ensuremath{\mathcal T(n,r)}}
\newcommand{\RTp} {\ensuremath{\mathcal T_{Po}(n,r)}}
\newcommand{\E}{\mathbb E}
\newcommand{\Prob}{\mathbb{P}}
\newcommand\eps{\varepsilon}

\newcommand{\Bb}{{\mathcal B}}
\newcommand{\Cc}{{\mathcal C}}
\newcommand{\Dd}{{\mathcal D}}
\newcommand{\Rr}{{\mathcal R}}

\usepackage{subfig}

\usepackage{enumerate}

\title{\scshape Localization Game for Random Geometric Graphs}

\author[1]{Lyuben Lichev}
\author[1,2]{Dieter Mitsche\footnote{Dieter Mitsche has been partially supported by grant Fondecyt grant 1220174 and by grant GrHyDy ANR-20-CE40-0002.}}
\author[3]{Pawe\l{} Pra\l{}at}
\affil[1]{Univ. Jean Monnet and Institut Camille Jordan, Saint-Etienne, France}
\affil[2]{Institute for Mathematical and Computational Engineering, PUC Chile}
\affil[3]{Toronto Metropolitan University, Toronto, Canada} 
\begin{document}

\maketitle
 
\begin{abstract}
The localization game is a two player combinatorial game played on a graph $G=(V,E)$. The cops choose a set of vertices $S_1 \subseteq V$ with $|S_1|=k$. The robber then chooses a vertex $v \in V$ whose location is hidden from the cops, but the cops learn the graph distance between the current position of the robber and the vertices in $S_1$. If this information is sufficient to locate the robber, the cops win immediately; otherwise the cops choose another set of vertices $S_2 \subseteq V$ with $|S_2|=k$, and the robber may move to a neighbouring vertex. The new distances to the robber are presented, and if the cops can deduce the new location of the robber based on all information they accumulated thus far, then they win; otherwise, a new round begins. If the robber has a strategy to avoid being captured, then she wins. The localization number is defined to be the smallest integer $k$ so that the cops win the game. In this paper we determine the localization number (up to poly-logarithmic factors) of the random geometric graph $G \in \RG$ slightly above the connectivity threshold.
\end{abstract}

\hspace{1em}Keywords: random graphs, random geometric graphs, localization number, cops and robbers

\hspace{1em}MSC Class: 05C80, 60C05, 05D40, 60D05

\section{Introduction}

\subsection{Localization game}

Graph searching focuses on the analysis of games and graph processes that model some form of intrusion in a network and efforts to eliminate or contain that intrusion. One of the best known examples of graph searching is the game of \emph{Cops and Robbers}, wherein a robber is loose on the network and a set of cops attempts to capture the robber. For a book on graph searching see~\cite{book}. 

In this paper we consider the \emph{Localization Game} that is related to the well studied \emph{Cops and Robbers} game. For a fixed integer $k\geq 1$, the localization game with $k$ sensors is a two player combinatorial game played on a graph $G$ which is known to both players. To initialize the game, the \emph{cops} first choose a set $S_1 \subseteq V(G)$ with $|S_1|=k$. The \emph{robber} then chooses a vertex $v\in V(G)$ to start at, whose location on the graph is hidden from the cops. The cops then learn the graph distance between the current position of the robber and the vertices of $S_1$. If this information is sufficient to locate the robber, then the cops win immediately. Otherwise, a new round begins, and the cops now choose another subset $S_2 \subseteq V(G)$ of size $k$, based on all the past information available to them. At this point, the robber is allowed to move to any vertex of distance one from its current location, based on $S_1$ \textit{and} $S_2$. The distances of the robber's new location to the vertices of $S_2$ are then presented to the cops, at which point the cops win if these new distance values in conjunction with the previous ones are sufficient to locate the robber. If the cops' information is still insufficient to win the game, then another round begins.  These rounds continue until the cops are able to locate the robber, in which case we say that the cops win, or the game proceeds indefinitely, in which case we say that the robber wins. Hence, to summarize, each round consists of the following \textit{steps}: 
\begin{enumerate}
    \item[a)] the cops place $k$ sensors on some vertices of $G$,
    \item[b)] the robber moves to a neighbor of the vertex she currently occupies or stays put (if this is the first round, then she simply chooses any vertex of $G$ to start with),
    \item[c)] the cops obtain the information about the distances between the sensors and the robber,
    \item[d)] the cops combine the information from all rounds so far and the game ends if this is enough to detect the position of the robber.
\end{enumerate}

We provide more details in Subsection~\ref{sec:perfect_information} to show that the localization game is a combinatorial game. This motivates the following definition. Given a graph $G$, its localization number, denoted by $\zeta(G)$, is the minimum $k$ such that the cops can eventually locate the robber using exactly $k$ sensors in each round. The localization game was introduced for one sensor ($k=1$) in~\cite{paper19, paper20} and was further studied in~\cite{Anthony_Bill, paper5, paper7, paper9, paper16}.

Let us emphasize that the cops only win provided their strategy beats all robber's strategies, and thus is a worst-case win condition. An alternative ``robber first'' definition of the localization game involves the robber moving first in each round, in particular choosing their move prior to the initial placement of the cops' sensors. Since both games require a worst case guarantee for the cops to win, these games are equivalent.

\subsection{Random geometric graphs}

In this paper we investigate geometric graphs in the plane. Given a positive integer $n$ and a threshold distance $r>0$, we consider the \emph{random geometric graph} $G \in \RG$ on vertex set $V=\{v_1, v_2, \ldots,v_n\}$ obtained by starting with $n$ random points $x_1, x_2, \ldots,x_n$ in $\R^2$ sampled independently and uniformly in the square $\left[0,\sqrt{n}\right]^2$. For any $i \neq j$, the vertices $v_i$ and $v_j$ are adjacent when the Euclidean distance $d_E(x_i, x_j)$ is at most $r$. Note that, with probability $1$, no point in $\left[0,\sqrt{n} \right]^2$ is chosen more than once, so we may identify each vertex $v_i \in V$ with its corresponding geometric position $x_i$. In fact, in order to simplify some of the proofs, we will work with the random geometric graph $G \in \RT$ equipped with the torus metric $d_T(\cdot, \cdot)$ instead of $d_E(\cdot, \cdot)$. For more details about these models see, for example,~\cite{Penrose}.

Our results are asymptotic in nature. In other words, we will assume that $n\to\infty$ and $r=r(n)$ may (and usually does) tend to infinity as $n\to\infty$. We are interested in events that hold \emph{asymptotically almost surely} (\emph{a.a.s.}), that is, events that hold with probability tending to 1 as $n\to \infty$. It is known that $r_c = r_c(n)=\sqrt{\frac{\log n}{\pi}}$ is a sharp threshold function for connectivity for $G \in \RG$ (see, for example, \cite{Goel05,Penrose97}). This means that for every $\varepsilon>0$, if $r\le(1-\varepsilon)r_c$, then $G$ is disconnected a.a.s., whilst if $r\ge (1+\varepsilon)r_c$, then $G$ is connected a.a.s. The same property holds for $G \in \RT$.

\subsection{Asymptotic notation}

Given two functions $f=f(n)$ and $g=g(n)$, we will write:
\begin{itemize}
    \item $f(n)=O(g(n))$ if there exists an absolute constant $c \in \R_+$ such that $|f(n)| \leq c|g(n)|$ for all $n$,
    \item $f(n)=\Omega(g(n))$ if $g(n)=O(f(n))$,
    \item $f(n)=\Theta(g(n))$ if $f(n)=O(g(n))$ and $f(n)=\Omega(g(n))$,
    \item $f(n)=o(g(n))$ or $f(n) \ll g(n)$ if $\lim_{n\to\infty} f(n)/g(n)=0$, and
    \item $f(n) \gg g(n)$ if $g(n)=o(f(n))$.
\end{itemize}

\subsection{Our main result}

Here and below, we fix $r_0 = r_0(n) = 70\sqrt{\log n}$. We split the statement of Theorem~\ref{thm:main} into four different cases corresponding to different proof strategies (or adaptations of proof strategies) for different values of the parameter $r$. Unfortunately, our proofs do not give insight which (if any) of the obtained bounds are tight.

\begin{theorem}\label{thm:main}
Fix $r = r(n) \in [r_0, \sqrt{n}/4)$ and let $G\in \RT$. Then, a.a.s.\ the following bounds hold: 
\begin{enumerate}
    \item If $\log^{3/2} n \le r < \sqrt{n}/4$, then $\Omega(r^{4/3}/ (\log n)^{1/3}) = \zeta(G) = O(r^{4/3})$.
    \item If $\log n \le r < \log^{3/2} n$, then $\Omega(r^{4/3}/ (\log n)^{1/3}) = \zeta(G) = O(\log^2 n)$. 
    \item If $\frac{\log n}{(\log \log n)^{1/2}\log \log \log n} \le r < \log n$, then $\Omega(r^2/\log n) = \zeta(G) = O(r^2)$.
    \item If $r_0 \le r < \frac{\log n}{(\log \log n)^{1/2}\log \log \log n}$, then $\Omega(\log n/\log(r^2/\log n)) = \zeta(G) = O(r^2)$.
\end{enumerate}
\end{theorem}

Note that the lower bound we prove below for $\tfrac{\log n}{(\log \log n)^{1/2}\log \log \log n} \le r < \log n$ in Theorem~\ref{thm:main} is slightly stronger, namely $\zeta(G)=\Omega(r^2 \log(e \log n/r)/\log n)$, but for the sake of readability we opted here for a slightly weaker version. Next, let us point out that we restrict ourselves to $r < \sqrt{n}/4$. This is done for a technical reason. For extremely dense graphs, the behaviour of $\RT$ changes drastically. In the extreme case, when $r \ge \sqrt{n/2}$, $\RT$ is simply the complete graph on $n$ vertices and $\zeta(\RT)=n-1$. Such dense graphs are not so interesting as they do not represent the typical nature of random geometric graphs and for that reason are rarely studied. Indeed, for such dense graphs the effect of wrapping around the torus has to be considered, and the results for $\RT$ typically differ from the ones for $\RG$, the random geometric graph in the square. 

\subsection{Main ideas behind the proofs}

The proof of Theorem~\ref{thm:main} in divided into a proof of the upper bounds and a proof of the lower bounds in the four regimes.
For the upper bounds, we first show that by using only four sensors the cops may locate the position of the robber throughout several rounds within a square $S$ of side length $20000 r$. Then, the cops need one last round to win. Roughly speaking, they divide their set of sensors into two parts of comparable sizes. Then, they distribute the first part of their sensors uniformly at random among all vertices of $G$ in the square $S$. Finally, the cops take care of the vertices in the square, which cannot be uniquely distinguished by the sensors already used, and put one sensor on any such vertex.\par

For the lower bounds, we show that it is sufficient for the robber to choose any ball of radius $r/3$ before even knowing the random graph $G\in \RT$. Once having done that, we prove that a.a.s.\ the number of sensors, given by Theorem~\ref{thm:main}, is not sufficient to distinguish the position of the robber even if she decides to stay in the ball forever.

\subsection{Related results}

The metric dimension of a graph $G$, written $\beta(G)$, is the minimum number of sensors needed in the localization game so that the cops can win in one round. The localization number is related to the metric dimension of a graph in a way that is analogous to how the cop number is related to the domination number. In particular, it follows that $\zeta(G) \le \beta(G)$, but in many cases this inequality is far from tight. 

Although the game has not yet been studied for random geometric graphs, there are some known results for the classical \emph{binomial random graph} $\mathcal{G}(n,p)$. The localization number for dense random graphs (in particular, in the regime in which $\mathcal G(n,p)$ has diameter two a.a.s.) was studied in~\cite{Dudek}. The bounds for dense graphs were consecutively improved in~\cite{Pawel}, and the arguments were extended to sparse graphs. 

The metric dimension was also studied for the $\mathcal{G}(n,p)$ model. The statements of the bounds for $\beta(G)$ with $G \in \mathcal{G}(n,p)$ obtained in~\cite{metric_dimension} are slightly technical, but the following observations can be made: for sparser graphs (that is, graphs of diameter at least three a.a.s., which corresponds to $i \ge 2$ in the discussion below), it follows from~\cite{metric_dimension} and~\cite{Pawel} that $\zeta(G) < \beta(G)$. In fact, if $np=n^{x+o(1)}$ for some $x\in (\frac{1}{i+1},\frac{1}i)$, $i \in \N \setminus \{1\}$, then a.a.s.\ $i + o(1) \le \beta(G) / \zeta(G) \le 1/x + o(1) < i+1$, and so these two graph parameters are a multiplicative constant away from each other (the ratio being roughly equal to the diameter of the graph). Moreover, for very sparse graphs, say for example $np = \log^6 n$, a.a.s.\ $\zeta(G) = \Theta( n \log \log n / (np)^i)$ whereas $\beta(G) = \Theta( n \log n / (np)^i)$, implying that for such value of $np$, $\zeta(G)=o(\beta(G))$. 

\section{Preliminaries}

\subsection{Reformulation of the game with perfect information for the cops}\label{sec:perfect_information}

In this section we show that the game we study is a combinatorial, perfect information game despite the fact that the robber is invisible for the cops. Let $G=(V,E)$ be a connected graph. Given a set $S\subseteq V$ of size $k$, $S=\{s_1,s_2,\dots,s_k\}$, and a vertex $v\in V$, the $S$-\emph{signature} of $v$ is defined as the vector $\mathbf{d}=\mathbf{d}(S,v)=(d_1,d_2,\dots,d_k)$ where for every $i \in \{1, 2, \ldots, k\}$, $d_i=d_G(s_i,v)$ is the graph distance from $s_i$ to $v$. Given a set $V' \subseteq V$, let 
$$
N[V'] = N_G[V'] := \{v \in V : d(v,u) \le 1 \text{ for some } u \in R \},
$$
that is, $N[V']$ is the closed neighborhood of the set of vertices $V'$ in $G$.\par

The \emph{localization game with $k$ sensors} is a game played by two players, the \emph{cops} and the \emph{invisible robber}. While playing the game, both the cops and the robber are aware of the underlying graph and each of the previous moves of the cops. However, the cops are not aware of the exact location of the robber while the robber is aware of every move they have made. Thus, the robber has perfect information in the localization game, but the cops do not, which at first sight contradicts our claim. Therefore, we propose the following reformulation of the game, which is based on a purely information theoretical perspective. When the cops put their sensors on the vertices of the set $S_1$, we partition the vertex set $V$ into $R^1_{1}\cup R^1_{2}\cup \ldots \cup R^1_{\ell_1}$ where the sets $(R^1_{j})_{1\leq j\leq \ell_1}$ are the equivalence classes of vertices in $V$ that have the same $S_1$-signature. Then, instead of choosing a specific location, the robber can choose some equivalence class $R^1_{j_1}$. Once the cops choose $S_2$, we partition the set $N[R^1_{j_1}]$ into equivalence classes $R^2_{1}\cup R^2_{2}\cup \ldots \cup R^2_{\ell_2}$ so that every vertex in $R^2_{j}$ has the same $S_2$-signature. Then, the robber chooses a set among $(R^2_{j_2})_{1\le j_2\le \ell_2}$. Iteratively, in round $i$, once the cops choose $S_i$, this gives the partition $N[R^{i-1}_{j_{i-1}}]=R^i_{1}\cup R^i_{2}\cup \ldots \cup R^i_{\ell_i}$ with every vertex in $R^i_{j}$ having the same $S_i$-signature; then the robber chooses some $R^i_{j_i}$. In this version of the game, the cops win in round $i$ if the robber is forced to choose a set $R^i_{j_i}$ with only one vertex, that is, $|R^i_{j_i}|=1$. In this reformulation, both players have perfect information. In particular, the localization game is a combinatorial game and so one of the players must have a winning strategy, that is, a strategy which wins against all of the other player's strategies simultaneously. We direct the reader to~\cite{Pawel} for a longer discussion. 

\subsection{Notation}

Let $\sim$ be the equivalence relation on $\R^2$ defined by $(0,x)\sim(\sqrt{n},x)$ and $(x,0)\sim(x,\sqrt{n})$ for every $x \in \R$. The \emph{torus $T_n$} is defined as $T_n = \R^2/\sim$ and is equipped with the natural metric $d_{T_n}$, inherited from the Euclidean metric $d_E$ on $\mathbb R^2$. To simplify notation, we write $d_T$ instead of $d_{T_n}$ below. The following definitions are used for both the Euclidean distance as well as the distance on the torus. For a given $x \in [0, \sqrt{n}]^2$ (respectively, $x\in T_n$) and $r \ge 0$, let $\Bb(x, r)$ be the \emph{(closed) ball with center $x$ and radius $r$}, that is, $\Bb(x, r) = \{y \in [0,\sqrt{n}]^2 : d_E(x,y) \le r\}$ (on $T_n$ we have $\Bb(x, r) = \{y \in T_n : d_T(x,y) \le r\}$). Let $\Cc(x,r)$ be the \emph{circle with center $x$ and radius $r$}, that is, $\Cc(x, r) = \{y \in [0,\sqrt{n}]^2 : d_E(x,y) = r\}$ (again, on $T_n$ we have $\Cc(x, r) = \{ y \in T_n : d_T(x,y) = r\}$). Finally, for $0\le r_1\le r_2$, let $\Dd(x, r_1, r_2) = \Bb(x, r_2)\setminus \Bb(x, r_1)$ be the \emph{crown with center $x$ and radii $r_1$ and $r_2$}. For any $d\ge 0$, we also use the term \emph{strip of width $d$} to denote the set of points in $\R^2$ at distance at most $d/2$ from a fixed line.

\medskip

As typical in the field of random graphs, we will use $\log x$ to denote the natural logarithm of $x$. Finally, for expressions that clearly have to be integer valued, we systematically round up or down without specifying which since the choice does not affect our arguments.

\subsection{De-Poissonization}\label{sec:Poisson}

In order to simplify some of our proofs, we will make use of a technique known as \emph{de-Poissonization}, which has many applications in geometric probability (see~\cite{Penrose} for a detailed account of the subject). Here we only roughly sketch the idea behind it.

Consider the following related models of random geometric graphs. Let $V=V'$, where $V'$ is a set given by a homogeneous Poisson point process of intensity $1$ in $[0,\sqrt{n}]^2$, respectively in $T_n$. In other words, $V'$ consists of $N$ points in the square $[0,\sqrt{n}]^2$, or in the torus $T_n$, chosen independently and uniformly at random, where $N$ is a Poisson random variable with expectation equal to $n$. By analogy to the models $\RG$ and $\RT$, almost surely no two vertices are located at the same position, and we are therefore allowed to identify any vertex $v_i$ with its geometric position $x_i$ in $[0,\sqrt{n}]^2$, respectively in $T_n$. Fix a parameter $r\ge 0$ and, for any pair of vertices $u$ and $v$ in $V'$, connect $u$ and $v$ if $d_E(u,v)\le r$, when working with $\mathcal G(N,r)$, and if $d_T(u,v)\le r$, when working with $\mathcal T(N,r)$. We denote these new models by $\RGp$ and $\RTp$.

Since our main result deals with the $\RT$ model, we concentrate on the connection between the models $\RT$ and $\RTp$. The same relationship holds for $\RG$ and $\RGp$. The main advantage of defining $V'$ via a Poisson point process is motivated by the following two properties: first, the number of vertices of $V'$ that lie in any measurable set $A\subseteq T_n$ of Lebesgue measure $a$ has a Poisson distribution with expectation $a$, and second, the number of vertices of $V'$ in disjoint subsets of $T_n$ are independently distributed. Moreover, by conditioning $\RTp$ upon the event $N=n$, we recover the original distribution of $\RT$. Therefore, since $\mathbb P(N=n)=\Theta(1/\sqrt n)$, any event holding in $\RTp$ with probability at least $1-o(f_n)$ must hold in $\RT$ with probability at least $1-o(f_n \sqrt n)$. 

We may also transfer results that hold in $\RT$ to $\RTp$. For example, suppose that for some random variable $X=X(G)$, there exist non-decreasing functions $f(n)$ and $g(n)$ such that a.a.s.\ $f(n) \le X \le g(n)$ for $G \in \RT$. Then, since a.a.s.\ $(1-\eps)n \le N \le (1+\eps)n$ for some $\eps = \eps(n) = o(1)$, we get that a.a.s.\ $f((1-\eps)n) \le X \le g((1+\eps)n)$ for $G \in \RTp$. In particular, our main result, Theorem~\ref{thm:main}, holds for $G \in \RTp$ as well.

\subsection{Concentration inequalities}

Let us first state a few specific instances of Chernoff's bound that we will find useful. Let $X \sim \textrm{Bin}(n,p)$ be a random variable distributed according to a Binomial distribution with parameters $n$ and $p$. Then, a consequence of \emph{Chernoff's bound} (see e.g.~\cite[Theorem~2.1]{JLR}) is that for any $t \ge 0$ we have
\begin{eqnarray}
\Prob( X \ge \E X + t ) &\le& \exp \left( - \frac {t^2}{2 (\E X + t/3)} \right)  \label{chern1} \\
\Prob( X \le \E X - t ) &\le& \exp \left( - \frac {t^2}{2 \E X} \right).\label{chern}
\end{eqnarray}

Moreover, let us mention that the bound holds in a more general setting as well, that is, for $X=\sum_{i=1}^n X_i$ where $(X_i)_{1\le i\le n}$ are independent variables and for every $i \in \{1, 2, \ldots, n\}$ we have $X_i \sim \textrm{Bernoulli}(p_i)$ with (possibly) different $p_i$-s (again, see~e.g.~\cite{JLR} for more details).

\bigskip

We will also need the following generalization of the previous bound due to Bentkus~\cite{Bentkus}, stated in a simplified form here. For two random variables $X$ and $Y$ defined on the same probability space, we write $X \preccurlyeq Y$ if $Y$ stochastically dominates $X$, that is, $\Prob(X \ge x) \le \Prob(Y \ge x)$ for all $x \in \mathbb{R}$. Let $\mathcal{L}(X)$ denote the distribution of the random variable $X$. For a positive random variable $Y$ and for $m > 0$ we define the random variable $Y^{[m]}$ so that $\E Y^{[m]}=m$, $Y^{[m]} \preccurlyeq Y$  and so that for some $b > 0$ we have $\Prob(0 < Y^{[m]}< b)=0$ and $\Prob(Y^{[m]} \ge x)=\Prob(Y \ge x)$ for all $x \ge b$ (in other words, one may roughly think of $Y^{[m]}$ as the random variable ``shifting mass that is close to $0$ to $0$ itself'').  

\begin{lemma}[\cite{Bentkus}]\label{bentkus}
Let $S=X_1+\ldots+X_{\ell}$ be a sum of $\ell$ positive independent random variables. Assume that for every $k \in \{1, 2, \ldots, \ell\}$ we have $X_k \preccurlyeq Y$ and $\E X_k \le m$ for some positive random variable $Y$ and some non-negative real number $m\le \mathbb E Y$. Let $T=\eps_1+\ldots+\eps_{\ell}$ be a sum of $\ell$ independent random variables $\eps_k$ so that $\mathcal{L}(\eps_k)=\mathcal{L}(Y^{[m]})$. Then, for all $x \in \mathbb{R}$, 
$$
\Prob(S \ge x) \le \inf_{h \le x} e^{-hx} \E e^{hT}.
$$
In particular, if $\E S \ge 1/c$ for some constant $c > 0$, 
$$
\Prob{ \left(S \ge c \, \E S\right)} \le e^{-c \E S} \E e^T.
$$
\end{lemma}

\subsection{Euclidean vs.\ graph distances}

Let us start with the following result from~\cite{DMPP-G}.

\begin{theorem}[\cite{DMPP-G}, Theorem 1.1 (ii)]\label{Thm DMPP-G}
Fix $r = r(n) \ge r_0$ and let $G\in \RG$. Then, a.a.s.\ for all pairs of vertices $u,v\in V(G)$ we have $d_G(u, v)\le \left\lceil\tfrac{d_E(u, v)}{r}(1 + \gamma r^{-4/3})\right\rceil$, where
\begin{equation}\label{eq:gamma}
\gamma = \max\left(31\left(\dfrac{2r\log n}{r+d_E(u,v)}\right)^{2/3}, \dfrac{70 \log^2 n}{r^{8/3}}, 300^{2/3}\right).
\end{equation}
\end{theorem}

As we plan to investigate $\RT$ instead of $\RG$, we need to adapt the above theorem to the torus metric. Fortunately, the adjustment is straightforward. 

\begin{corollary}\label{cor:distances}
Fix $r = r(n) \ge r_0$ and let $G\in \RT$. Then, a.a.s.\ the following property holds for all pairs of vertices $u,v\in V(G)$:
$$
d_G(u, v)\le \left\lceil\dfrac{d_T(u, v)}{r}(1 + \gamma r^{-4/3})\right\rceil,
$$
where $\gamma$ is defined in~\eqref{eq:gamma}.
\end{corollary}

\begin{proof}
We generate $\RT$ and nine copies of $\RG$ that will be coupled in the following way. Start with $n$ random points $x_1, x_2, \ldots,x_n$ in the square $\left[0,\sqrt{n}\right]^2$ sampled independently and uniformly. We use these points to generate $G \in \RT$. We stay with these $n$ points on the torus, and then translate our $\sqrt{n}\times \sqrt{n}$-window by the vector $(i \sqrt{n}/2, j \sqrt{n}/2)$ for some $i,j \in \{-1,0,1\}$; in other words, we consider the square $[i \sqrt{n}/2, \sqrt{n} + i \sqrt{n}/2] \times [j \sqrt{n}/2, \sqrt{n} + j \sqrt{n}/2]$. In fact, for example, the squares corresponding to $(i,j)=(-1,-1)$ and to $(i,j)=(1,1)$ coincide but it will be convenient to keep 9 squares instead of~4. Indeed, for any two points $u,v$ in the original square, the toroidal distance between $u$ and $v$ is the minimum distance between $u$ (taken in the original square) and all 9 images of the vertex $v$ under the above translations. Each of these 9 choices yields one copy of $G_{ij} \in \RG$. 

Since we aim for a statement that holds a.a.s., we may assume that for each $G_{ij}$, the statement of Theorem~\ref{Thm DMPP-G} is satisfied. Since we have 10 graphs and 9 squares (one graph for each of the images of the square $[0, \sqrt{n}]^2$ under the above translations, and the graph on $T_n$), we will use superscripts to indicate which graph/square we consider. Consider any pair of vertices $u,v\in V(G)$. Clearly, for some $G_{ij}$ we have $d_T(u, v) = d_E^{G_{ij}}(u, v)$. Indeed, the shortest segment $uv$ in $T_n$ is contained in some square of side length $\sqrt{n}/2$, and any such square is contained in some of the nine squares $([i \sqrt{n}/2, \sqrt{n} + i \sqrt{n}/2] \times [j \sqrt{n}/2, \sqrt{n} + j \sqrt{n}/2])_{i,j\in \{-1,0,1\}}$. Also, since $G_{ij}$ is a subgraph of $G$, $d_G^G(u, v) \le d_G^{G_{ij}} (u, v)$.  Combining these observations together we get that
$$
d_G^G(u, v) \le d_G^{G_{ij}} (u, v) \le \left\lceil\dfrac{d_E^{G_{ij}}(u, v)}{r}(1 + \gamma r^{-4/3})\right\rceil = \left\lceil\dfrac{d_T(u, v)}{r}(1 + \gamma r^{-4/3})\right\rceil.
$$
The proof of the corollary is finished. 
\end{proof}

We will also need the following simple but useful observation.
\begin{observation}\label{obs:vertices}
Let $G\in \RT$. Then, a.a.s.\ for any point $x \in T_n$ there exists a vertex $v_i \in V(G)$ such that $d_T(x,v_i) \le 2 \sqrt{\log n}$.
\end{observation}
\begin{proof}
Fix $k = \lfloor\sqrt{n / \log n} / 1.1\rfloor$. Tessellate $T_n$ into $k^2$ small squares, each of side length $\sqrt{n} / k = (1.1+o(1)) \sqrt{\log n}$. The probability that a given small square contains no vertex is equal to
$$
\left( 1 - \frac {(\sqrt{n} / k)^2 }{n} \right)^n \le \exp \Big(- (\sqrt{n} / k)^2 \Big) = \exp \Big( - (1.21+o(1)) \log n \Big) = o(n^{-1}).
$$
Since there are $k^2 = o(n)$ small squares, it follows from the union bound over all small squares that a.a.s.\ each small square contains a vertex. Since we aim for a statement that holds a.a.s., we may assume that this property holds and then the conclusion follows deterministically. Indeed, since $1.1 \sqrt{2} < 2$, for any point $x \in T_n$ the ball $\Bb(x,2 \sqrt{\log n})$ contains at least one square, which implies the result. 
\end{proof}

\section{Upper bound}

This section is devoted to the proof of the upper bounds stated in Theorem~\ref{thm:main}.

\bigskip

Let us start by showing that the cops are able to localize the robber within a square of side length $20000 \, r$ by using only four sensors. We prepare the ground with the following lemma.

\begin{lemma}\label{lem 1}
Fix $r = r(n) \ge r_0$ and let $G\in \RT$. 
Suppose that $\RT$ satisfies the properties stated in Corollary~\ref{cor:distances} and Observation~\ref{obs:vertices}.
Let $s = s(n)$ be such that $20000 \, r\le s \le \sqrt{n}/9$. Suppose that at the beginning of some round the robber occupies a vertex inside a square $S$ of side length $s$ and at Euclidean distance at least $r$ from the border of $S$. Then, the cops may place four sensors so that at the end of the current round the robber is localized within a square of side length $s/4$ and at distance at least $r$ from the border of this square. 
\end{lemma}

\begin{proof}
Consider four points $A,B,C,D \in T_n$ that are the four corners of the square $S'$ of side length $3s$, with sides parallel to the sides of $S$, and containing the square $S$ in its center. By our assumption, $\max_{u,v\in S'} d_T(u,v)\le 3\sqrt{2}s \le \sqrt{2n}/3 < \sqrt{n}/2$ (note that $S$ is not necessarily axis-parallel); in particular, the geodesic between any two points in $S'$ is included in $S'$, that is, $S'$ is small enough so that the metric $d_T$ on $T_n$ coincides with the Euclidean metric on the square $S'$. Place sensors at the vertices $v_A, v_B, v_C, v_D$ that are the closest to $A,B,C,D$ (not necessarily in the square $ABCD$), respectively. By Observation~\ref{obs:vertices}, we may find a vertex within Euclidean distance $2 \sqrt{\log n}$ for any choice of points $A,B,C,D$. Let $d_A, d_B, d_C, d_D$ be the graph distances from sensors $v_A, v_B, v_C, v_D$, respectively, to the robber once she makes her move (that is, after step b) of the current round). By our assumption, she is still inside the square $S$.

Now, by Corollary~\ref{cor:distances} and the fact that for all $i\in \{A,B,C,D\}$ we have $r d_G(v_i,R)\ge d_T(v_i, R) = d_E(v_i, R)$, the robber is in the crown 
\begin{equation*}
    \Dd\left(v_i, \dfrac{r (d_i-1)}{1+\gamma r^{-4/3}}, rd_i\right) = \Bb(v_i, rd_i)\setminus \Bb\left(v_i, \dfrac{r (d_i-1)}{1+\gamma r^{-4/3}}\right).
\end{equation*}
Note that the Euclidean distance between any point $i\in\{A,B,C,D\}$ and the position of the robber after she moves is at least $d_E(i, S) = \sqrt{2} s$ (see Figure~\ref{fig 1}). Hence, since $r\ge r_0$ and $s\ge 20000 \, r$, we get that for every $i\in\{A,B,C,D\}$, 
$$
d_i \ge \left\lceil \frac {\sqrt{2} s - 2\sqrt{\log n} }{r} \right\rceil \ge 20000.
$$ 
Moreover, $\gamma r^{-4/3}\le 1/50$: indeed, we have
\begin{eqnarray*}
    31\left(\dfrac{2r\log n}{r+d_E(u,v)}\right)^{2/3}r^{-4/3} &\le& 31\left(\dfrac{2r\log n}{20000r}\right)^{2/3} 70^{-4/3} \log^{-2/3} n = \dfrac{31}{10^{8/3} 70^{4/3}} < \dfrac{1}{50},\\ 
    \dfrac{70 \log^2 n}{r^{8/3}} \, r^{-4/3} &=& \dfrac{70 \log^2 n}{r^4}\le \dfrac{70}{70^4} < \dfrac{1}{50}, \text{ and} \\
    300^{2/3} r^{-4/3} &\ll& \dfrac{1}{50}.
\end{eqnarray*}
Therefore, each of the four crowns 
\begin{equation*}
    \left(\Dd\left(v_i, \dfrac{r (d_i-1)}{1+\gamma r^{-4/3}}, rd_i\right)\right)_{1\le i\le 4}
\end{equation*}
has width 
\begin{equation*}
    rd_i - \dfrac{r(d_i-1)}{1+\gamma r^{-4/3}}\le rd_i/51 + 50r/51\le (3/50 + 1/20000) s \le s/6 - 4 \sqrt{\log n}.
\end{equation*} 
The first and the third inequalities follow from a direct computation, while the second inequality uses the fact that 
\begin{equation*}
    rd_i\le \frac{51}{50} d_E(v_i, R) + r\le \frac{51}{50} (d_E(i, R) + 2\sqrt{\log n}) + r \le \frac{51}{50} (2\sqrt{2} s + r) + r \le \frac{51}{50}\cdot 3s.
\end{equation*}
Since for every $i\in \{A,B,C,D\}$ we have $d_E(i, v_i)\le 2 \sqrt{\log n}$, we get that for any radius $\rho\ge 0$ we have 
\begin{equation*}
   \Dd(v_i, \rho + 2 \sqrt{\log n}, \rho + s/6 - 2 \sqrt{\log n}) \subseteq \Dd(i, \rho, \rho + s/6),
\end{equation*}
so the robber must be hiding inside
\begin{equation*}
    \bigcap_{i\in \{A,B,C,D\}} \Dd(v_i, \rho_i + 2 \sqrt{\log n}, \rho_i + s/6 - 2 \sqrt{\log n})\subseteq \bigcap_{i\in \{A,B,C,D\}} \Dd(i, \rho_i, \rho_i + s/6),
\end{equation*}
where 
\begin{equation}\label{eq:rhoi}
\rho_i = \dfrac{r (d_i-1)}{1+\gamma r^{-4/3}} - 2 \sqrt{\log n}.
\end{equation}

\bigskip 

It remains to show that the four crowns with centers at the corners of the square $ABCD$ intersect in a region, which is contained in a square of side at most $s/4 - 2r$. The next purely geometric claim is the key to our proof of this fact.

\begin{sublemma}\label{claim 1}
Let $\rho_B, \rho_D > 0$ be such that
\begin{equation*}
    \Dd(B, \rho_B, \rho_B + s/6)\cap \Dd(D, \rho_D, \rho_D + s/6)\cap S\neq \emptyset.
\end{equation*}
Then, $\Dd(B, \rho_B, \rho_B + s/6)\cap \Dd(D, \rho_D, \rho_D + s/6)$ is included in a strip, parallel to the diagonal $AC$ and of width at most $s/4 - 2r$.
\end{sublemma}
\begin{proof} 
If $\rho_B+\rho_D+s/6\le |BD|$, then one may easily conclude that 
\begin{equation*}
    \Dd(B, \rho_B, \rho_B + s/6)\cap \Dd(D, \rho_D, \rho_D + s/6)
\end{equation*}
is included in a strip between two lines, parallel to $AC$ and at distance at most $s/6 \le s/4 - 2r$.\par

Otherwise, let $\Cc(B, \rho_B)\cap \Cc(D, \rho_D+s/6) = \{Q', Q''\}$ with $Q', A$ on the same side with respect to $BD$, $\Cc(B, \rho_B+s/6)\cap \Cc(D, \rho_D+s/6) = \{P', P''\}$ with $P', A$ on the same side with respect to $BD$, and $\Cc(B, \rho_B+s/6)\cap \Cc(D, \rho_D) = \{R', R''\}$ with $R', A$ on the same side with respect to $BD$. We know that $Q'Q'' ~||~ P'P'' ~||~ R'R''$, and the three of them are parallel to $AC$. Also, define $P,Q$, and $R$ as the intersection points of $BD$ with the segments $P'P'', Q'Q''$ and $R'R''$, respectively. See Figure~\ref{fig 1} for an illustration.
By the Pythagorean theorem 
\begin{equation*}
    |DP|^2 - |PB|^2 = |DP''|^2 - |P''B|^2 = (\rho_D+s/6)^2 - (\rho_B+s/6)^2,
\end{equation*}
and
\begin{equation*}
    |DQ|^2 - |QB|^2 = |DQ''|^2 - |Q''B|^2 = (\rho_D+s/6)^2 - \rho_B^2.
\end{equation*}
We conclude that
\begin{eqnarray*}
    s^2/36 + \rho_B s/3 &=& (|DQ|^2 - |QB|^2) - (|DP|^2 - |PB|^2)\\ &=& (|DQ|-|QB|)|DB| - (|DP|-|PB|)|DB| \\
    &=& 2\cdot |PQ|\cdot |DB| = 6\sqrt{2}s\cdot |PQ|.
\end{eqnarray*}
Since $\rho_B\le (2\sqrt{2}+1/6)s$ (recall that $\Dd(B, \rho_B, \rho_B + s/6)\cap S\neq \emptyset$), we have that 
\begin{equation*}
    |PQ| \le \dfrac{1}{6\sqrt{2}}\left(\dfrac{2\sqrt{2}}{3} + \dfrac{1}{6\cdot 3} + \dfrac{1}{36}\right)s < 0.121\cdot s.
\end{equation*}
A similar argument implies that 
\begin{equation*}
    |QR| \le \dfrac{1}{6\sqrt{2}}\left(\dfrac{2\sqrt{2}}{3} + \dfrac{1}{6\cdot 3} + \dfrac{1}{36}\right)s < 0.121\cdot s.
\end{equation*}

\begin{figure}
\centering
\begin{tikzpicture}[line cap=round,line join=round,x=1cm,y=1cm]
\clip(-4.5,-3.25) rectangle (5.5,6.25);
\draw [line width=0.8pt] (-4,6)-- (-4,-3);
\draw [line width=0.8pt] (-4,-3)-- (5,-3);
\draw [line width=0.8pt] (-4,6)-- (5,-3);
\draw [line width=0.8pt] (5,-3)-- (5,6);
\draw [line width=0.8pt] (5,6)-- (-4,6);
\draw [line width=0.8pt] (-1,3)-- (-1,0);
\draw [line width=0.8pt] (-1,0)-- (2,0);
\draw [line width=0.8pt] (2,0)-- (2,3);
\draw [line width=0.8pt] (2,3)-- (-1,3);
\draw [line width=0.8pt] (5,-3) circle (7.5cm);
\draw [line width=0.8pt] (-4,6) circle (8cm);
\draw [line width=0.8pt] (-4,6) circle (6cm);
\draw [line width=0.8pt] (5,-3) circle (5.5cm);
\draw [line width=0.8pt] (-2.409807097612819,-1.8403626531683746)-- (3.8403626531683748,4.409807097612819);
\draw [line width=0.8pt] (-1.6424030824705673,0.4825969175294327)-- (1.5174030824705673,3.642403082470567);
\draw [line width=0.8pt] (-0.12249799679358563,-0.9974979967935856)-- (2.9974979967935855,2.1224979967935855);
\begin{scriptsize}
\draw [fill=black] (-4,6) circle (1.5pt);
\draw[color=black] (-4.23162373981443,6.10) node {$D$};
\draw [fill=black] (-4,-3) circle (1.5pt);
\draw[color=black] (-4.23162373981443,-2.85) node {$A$};
\draw [fill=black] (5,-3) circle (1.5pt);
\draw[color=black] (5.187381667244507,-2.85) node {$B$};
\draw [fill=black] (5,6) circle (1.5pt);
\draw[color=black] (5.187381667244507,6.10) node {$C$};

\draw [fill=black] (5,4.5) circle (1.5pt);
\draw [fill=black] (4,6) circle (1.5pt);
\draw [fill=black] (2,6) circle (1.5pt);
\draw [fill=black] (5,2.5) circle (1.5pt);
\draw [fill=black] (-2.409807097612819,-1.8403626531683746) circle (1.5pt);
\draw[color=black] (-2.6021755470320818,-1.5856379557122442) node {$P'$};
\draw [fill=black] (3.8403626531683748,4.409807097612819) circle (1.5pt);
\draw[color=black] (3.677161391007209,4.6539563434787095) node {$P''$};
\draw [fill=black] (-1.6424030824705673,0.4825969175294327) circle (1.5pt);
\draw[color=black] (-1.7940752237822999,0.7326826437748406) node {$R'$};
\draw [fill=black] (1.5174030824705673,3.642403082470567) circle (1.5pt);
\draw[color=black] (1.3720883378029123,3.8723511127944925) node {$R''$};
\draw [fill=black] (-0.12249799679358563,-0.9974979967935856) circle (1.5pt);
\draw[color=black] (-0.2971024938277853,-0.7642900861796769) node {$Q'$};
\draw [fill=black] (2.9974979967935855,2.1224979967935855) circle (1.5pt);
\draw[color=black] (2.7630806974951607,2.3621308365571916) node {$Q''$};
\draw [fill=black] (-0.0625,2.0625) circle (1.5pt);
\draw[color=black] (-0.3235975863933519,2.1) node {$R$};
\draw [fill=black] (0.7152777777777779,1.284722222222222) circle (1.5pt);
\draw[color=black] (0.431512551725297,1.3288222265000909) node {$P$};
\draw [fill=black] (1.4375,0.5625) circle (1.5pt);
\draw[color=black] (1.1601275972783793,0.6) node {$Q$};
\end{scriptsize}
\end{tikzpicture}
\caption{Illustration for the proof of Claim~\ref{claim 1}.}
\label{fig 1}
\end{figure}
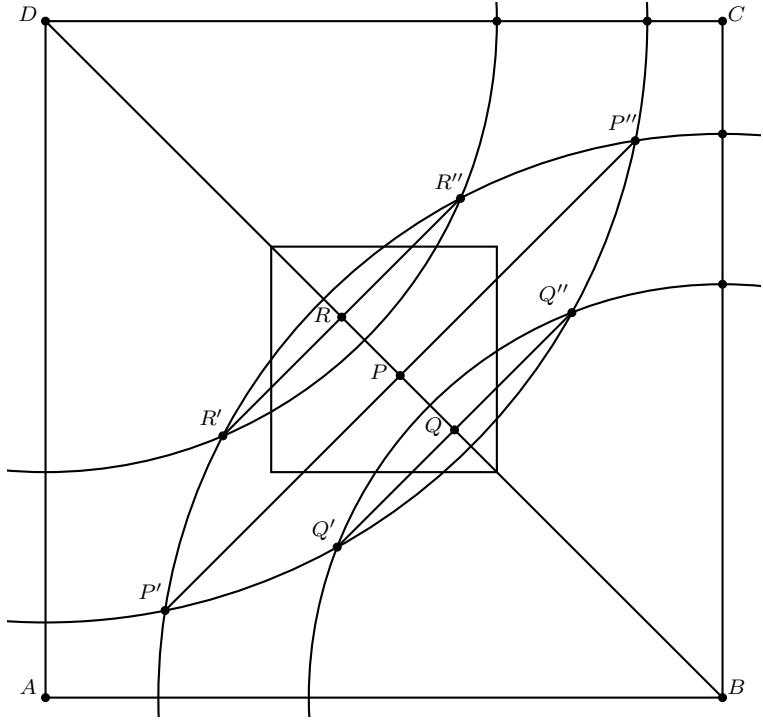

\noindent Thus, the strip between the lines $Q'Q''$ and $R'R''$ contains $\Dd(B, \rho_B, \rho_B + s/6)\cap \Dd(D, \rho_D, \rho_D + s/6)$, and the distance between these two lines is given by $|RQ| = |RP| + |PQ| < 2\cdot 0.121\cdot s = 0.242\cdot s < s/4 - 2r$. The proof of the claim is finished.
\end{proof}

Applying Claim~\ref{claim 1} to the intersection 
\begin{equation*}
\Dd\left(A, \rho_A, \rho_A + s/6 \right) \cap \Dd\left(C, \rho_C, \rho_C + s/6 \right),
\end{equation*}
and to the intersection
\begin{equation*}
\Dd\left(B, \rho_B, \rho_B + s/6 \right) \cap \Dd\left(D, \rho_D, \rho_D + s/6 \right),
\end{equation*}
we get that the intersection of all these four crowns is contained in a square of side $s/4 - 2r$. This square is situated in the center of a larger square with sides parallel to the sides of the smaller square, and of length $s/4$, which finishes the proof of the lemma.
\end{proof}

Now, we put all observations together and show that, throughout several rounds, the cops are able to localize the robber within a square of side length $20000 \, r$ by using only four sensors.

\begin{corollary}\label{cor 4 cops}
Using only four sensors, a.a.s.\ the robber can be localized on $\RT$ within a square of side length $20000 \, r$.
\end{corollary}
\begin{proof}
If $r \ge \sqrt{n}/20000$, then there is nothing to prove. Hence, suppose that $r \le \sqrt{n}/20000$. Since we aim for a statement that holds a.a.s., we may assume that $\RT$ satisfies the properties stated in Corollary~\ref{cor:distances} and Observation~\ref{obs:vertices}.

Let $B$ be the set of the $400$ vertices of the $20 \times 20$ square grid of mesh size $\sqrt{n}/20$, covering $T_n$. Construct a set $C$ by adding, for every vertex $h$ of $B$, a vertex of $G$ at distance at most $2\sqrt{\log n}$ from $h$ (the existence of such vertex is guaranteed by Observation~\ref{obs:vertices}; if there are more choices for a given $h$, then we may choose arbitrarily). Let the cops put sensors at the vertices of $C$ in groups of 4, one group after another, so that all vertices are tested in 100 rounds. Trivially, from the first test to the last one, the robber changes her position by at most $100 r$. 

Let $u \in C$ be a vertex that detected the closest graph distance to the robber (if there are many such vertices in $C$, then we select one of them arbitrarily). Our goal is to estimate the Euclidean distance (coinciding with the distance on the torus $T_n$) from $u$ to the robber once testing is finished (that is, after 100 initial rounds). Note that the robber had to be initially at distance at most $\sqrt{2n}/40$ from some point in $B$ and so at distance at most $\sqrt{2n}/40 + 2\sqrt{\log n} = (\sqrt{2}/40 + o(1)) \sqrt{n}$ from some vertex $w$ in $C$. Hence, she is certainly at distance at most $(\sqrt{2}/40 + o(1)) \sqrt{n} + 100r < \sqrt{n}/ 24$ from $w$ when $w$ was probed. More importantly, by Corollary~\ref{cor:distances} (and the computations done in the proof of Lemma~\ref{lem 1}), we know that at that point of the game the graph distance from $w$ to the robber was at most
\begin{equation*}
    \dfrac{51}{50} \cdot \dfrac{\sqrt{n}}{24} \cdot \dfrac{1}{r} + 1.
\end{equation*}
As a result, since $u$ is the sensor that returned the smallest graph distance, the graph distance from $u$ to the position she was when $u$ was probed is at most 
\begin{equation*}
    \dfrac{51}{50} \cdot \dfrac{\sqrt{n}}{24} \cdot \dfrac{1}{r} + 1,
\end{equation*}
and so the Euclidean distance between $u$ and the position of the robber at the end of round 100 is at most
\begin{equation*}
    \dfrac{51}{50} \cdot \dfrac{\sqrt{n}}{24} + r + 100 r < \sqrt{n} / 20.
\end{equation*}
Hence, the cops have a strategy to find a square of side length $\sqrt{n}/10$ in which the robber is located at the end of round $100$. By making the square slightly larger (that is, of side length $\sqrt{n}/9$), we are guaranteed that she is at distance at least $r$ from the border. Finally, we may consecutively apply Lemma~\ref{lem 1} to get the desired property and finish the proof.
\end{proof}

\bigskip

At this point of the game, we may assume that the robber occupies a vertex in a region $\mathcal R$ that is inside a square of side length $20000 \, r$. For a region $\mathcal R$, we define $N[\mathcal R] \subseteq V(G)$ as the subset of the vertices of $G$ contained in the union of all balls of radius $r$, centered at the vertices in $V(G)\cap \mathcal R$. The cops aim to finish the game in the very next round by choosing a set $W$ of vertices to put sensors on such that, regardless where the robber moves, she is going to be localized. In other words, their goal is to partition $N[\mathcal R]$ into equivalence classes with the same $W$-signature such that each class consists of a single vertex (see Subsection~\ref{sec:perfect_information} for a convenient reformulation of the game that explains this line of thinking). In this case, we also say that the set of sensors $W$ \emph{distinguishes} the vertices in the set $N[\mathcal R] \subseteq V(G)$. Trivially, $N[\mathcal R]$ is contained in a square $S$ of side length $20002 \, r$. Of course, the robber plays the game optimally so she can try to ``get trapped'' in a region $\mathcal R$ that is placed in some convenient (for her) part of the square $[0,\sqrt{n}]^2$. Hence, we need to show that, regardless what she does, she will suffer the same fate and lose the game in the very next round.

Let $\mathcal{F}$ be a family of squares of side length $10^5 \, r$, with sides parallel to the axes, and with left-bottom vertices at points $(10^4 \, r \, i, 10^4 \, r \, j)$ for some $i,j \in \N \cup \{0\}$ such that $10^4 \, r \, i < \sqrt{n}$ and $10^4 \, r \, j < \sqrt{n}$. Clearly, $|\mathcal{F}| = O(n / r^2) < n$. For a given square $S \in \mathcal{F}$, let $I(S)$ be defined as the square of side length $10^5 r - 2r$ inside $S$, centered at the same point as $S$ and with sides, parallel to the sides of $S$. We call $I(S)$ the \emph{internal square} of $S$. Clearly, $N[\mathcal R] \subseteq I(S) \subseteq S$ for some $S \in \mathcal{F}$. Hence, in order to finish the proof of the upper bound, it remains to show the following lemma.

\begin{lemma}\label{lem:second_phase}
Fix $r = r(n) \in [r_0,\sqrt{n}/4)$ and let $G\in \RT$. Let 
$$
w = w(n) = 
\begin{cases}
10^{15} \, r^{4/3} & \text{ if } r \ge \log^{3/2} n, \\
3 \cdot 10^{16} \log^2 n & \text{ if } 100 \log n \le r < \log^{3/2} n, \\
2 \cdot 10^{10} r^2 & \text{ if } r < 100 \log n. \\
\end{cases}
$$
Then, a.a.s.\ the following property holds: for any square $S \in \mathcal{F}$, there exists a set of vertices $W=W(S) \subseteq S \cap V(G)$ of cardinality at most $w$ such that placing sensors on $W$ distinguishes all vertices in the internal square $I(S)$, that is, all vertices in $I(S)$ have a unique $W$-signature.
\end{lemma}
 
Before diving into the proof of Lemma~\ref{lem:second_phase}, we provide a rough sketch of the main idea. Let us fix $\delta =\delta(n) = o(1)$ (to be chosen appropriately later on). For every square $S \in \mathcal{F}$, the set $W=W(S)$ is constructed as follows: we investigate all vertices in $S$ and independently put them into a set $X$ with probability $\delta$. This set partitions the vertices in the internal square $I(S)$ into equivalence classes with the same $X$-signature. We do not expect each class to contain only one vertex so we investigate all equivalence classes. If some class contains at least two vertices, then we put all vertices from that class into a set $Y$. (In fact, we may put all but one of them into $Y$ but, for simplicity, we include all of them as it would not improve the asymptotic order of the bound.) By construction, the set $W = X \cup Y$ achieves the desired goal of identifying the robber since each non-sensor vertex in $I(S)$ has a unique $X$-signature (otherwise, it would be put into $Y$) and so also a unique $W$-signature, and each sensor vertex in $I(S)$ has a unique $W$-signature as it is the only vertex at distance 0 from itself. Note that, roughly speaking, if $X$ is small, then $Y$ has to be large and vice versa. Hence, at some point we will have to optimize $\delta$ as a function of $r$ since we aim to find a set $W$, which is as small as possible.

\bigskip

In the next observation we investigate $\Bb(A,r)\triangle \Bb(B,r)$, the symmetric difference of two discs centered in $A$ and $B$. We show a lower bound for the area of this symmetric difference which is a non-decreasing function of the distance between $A$ and $B$.

\begin{observation}\label{ob AB}
Fix $r = r(n) < \sqrt{n}/4$ and let $A,B$ be any two points in $T_n$ at distance $\varepsilon$ from each other. If $\eps \le \eps_0 := 2r$, then
\begin{equation*}
| \Bb(A,r)\triangle \Bb(B,r) | = (2\pi - 4\arccos(\varepsilon/2r)) r^2 + 2\varepsilon r \sqrt{1 - \dfrac{\varepsilon^2}{4r^2}} \ge 2\eps r.   
\end{equation*}
In particular, if $\eps \ll r$, then 
\begin{equation*}
| \Bb(A,r)\triangle \Bb(B,r) | = (4+o(1)) \eps r.  
\end{equation*}
On the other hand, if $\eps > \eps_0$, then trivially
$$
| \Bb(A,r)\triangle \Bb(B,r) | = | \Bb(A,r) | + | \Bb(B,r) | = 2 \pi r^2 \ge 2 \eps_0 r.  
$$
\end{observation}
\begin{proof}

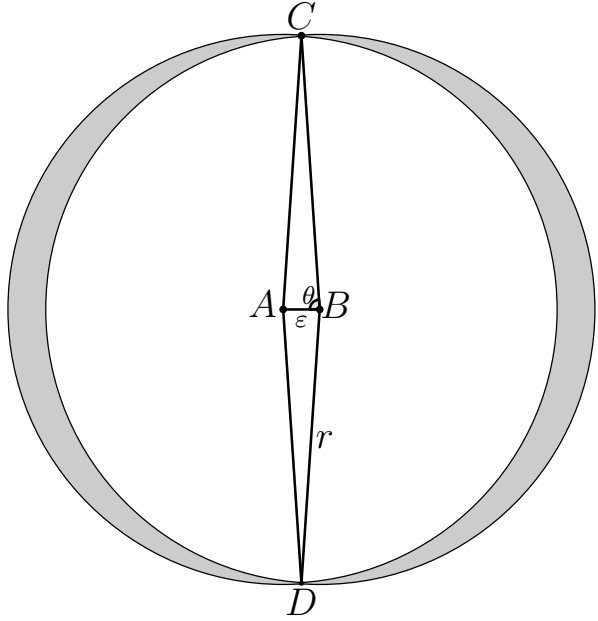
\begin{figure}
\centering
\resizebox{25em}{!}{%
\begin{tikzpicture}[line cap=round,line join=round,x=1cm,y=1cm,scale=0.4]
\clip(-9.7,-9) rectangle (10,8);

\draw [shift={(0,-0.5)},line width=0.8pt,color=black,fill=black,fill opacity=0.10000000149011612] (0,0) -- (93.81407483429037:0.2703580656096896) arc (93.81407483429037:180:0.2703580656096896) -- cycle;

\draw [line width=0.8pt] (-1,-0.5) circle (7.516648189186455cm);
\draw [line width=0.8pt] (0,-0.5) circle (7.516648189186455cm);

\fill[gray!40,even odd rule] (0,-0.5) circle (7.516648189186455cm) (-1,-0.5) circle (7.516648189186455cm);

\draw [line width=0.8pt] (-1,-0.5)-- (0,-0.5);
\draw [line width=0.8pt] (-1,-0.5)-- (-0.5,7);
\draw [line width=0.8pt] (-0.5,7)-- (0,-0.5);
\draw [line width=0.8pt] (-1,-0.5)-- (-0.5,-8);
\draw [line width=0.8pt] (0,-0.5)-- (-0.5,-8);
\begin{scriptsize}
\draw [fill=black] (-1,-0.5) circle (2.5pt);
\draw[color=black] (-1.5623813401027786,-0.35512488171559686) node {\large{$A$}};
\draw [fill=black] (-0.5,7) circle (2.5pt);
\draw[color=black] (-0.5,7.53) node {\large{$C$}};
\draw [fill=black] (0,-0.5) circle (2.5pt);
\draw[color=black] (0.438268345408925,-0.3821606882765659) node {\large{$B$}};
\draw[color=black] (-0.5,-0.82) node {$\varepsilon$};
\draw [fill=black] (-0.5,-8) circle (1.5pt);
\draw[color=black] (-0.5,-8.5) node {\large{$D$}};
\draw[color=black] (0.14087447323826635,-4.0590303805683545) node {\large{$r$}};
\draw[color=black] (-0.3,-0.1) node {$\theta$};
\end{scriptsize}
\end{tikzpicture}}  
\caption{An illustration from the computation in \eqref{eq crown}.}
\label{fig 2}
\end{figure}

Since $r < \sqrt{n}/4$ we have that $\Bb(A,r)\cap \Bb(B,r)$ is a (possibly empty) connected subset of $T_n$. Let $\Cc(A,r) \cap \Cc(B,r) = \{C,D\}$, $\angle CBA = \angle ABD = \theta$, and let $\varepsilon$ be the distance between $A$ and $B$. Suppose that $\eps \le 2r$ since the statement for larger values of $\eps$ is trivial. (An illustration of the configuration may be found on Figure~\ref{fig 2}.) The area of $\Bb(A,r)\triangle \Bb(B,r)$ (the grey region in Figure~\ref{fig 2}) is by a simple inclusion-exclusion formula equal to
\begin{equation*}
    2 \, \dfrac{2\pi - 2\theta}{2\pi}\pi r^2 - 2\, \dfrac{2\theta}{2\pi} \pi r^2 + 2\cdot 2 \, \dfrac{\varepsilon r \sin\theta}{2} = \left(2\pi - 4\arccos\left(\dfrac{\varepsilon}{2r}\right)\right) r^2 + 2\varepsilon r \sqrt{1 - \dfrac{\varepsilon^2}{4r^2}}.
\end{equation*}
The desired bound holds since $\arccos(\eps/2r) \le \pi/2 - \eps/2r$.

If we additionally suppose that $\varepsilon\ll r$, then $\sin(\theta) = 1-o(1)$ and the above equality can be simplified as follows:
\begin{equation}\label{eq crown}
    2 \, \dfrac{2\pi - 2\theta}{2\pi}\pi r^2 - 2\, \dfrac{2\theta}{2\pi} \pi r^2 + (4+o(1)) \dfrac{\varepsilon r}{2} = (2\pi - 4\theta) r^2 + (2+o(1)) \varepsilon r.
\end{equation}
Moreover, $\cos(\theta) = \varepsilon/2r = o(1)$ and thus
\begin{equation*}
\theta = \arccos(\eps/2r) = \dfrac{\pi}{2} - \dfrac{\eps}{2r} - O\left(\dfrac{\eps^3}{8r^3}\right) =  \dfrac{\pi}{2} - (1+o(1))\dfrac{\eps}{2r},
\end{equation*}
and so 
\begin{equation*}
|\Bb(A,r)\triangle \Bb(B,r)| = (2\pi - 4\theta) r^2 + (2+o(1)) \varepsilon r = (4+o(1))\varepsilon r.
\end{equation*}
The proof of the observation is finished.
\end{proof}

Observation~\ref{ob AB} is enough to show that, for any two points $A$ and $B$ on the torus $T_n$, with high probability there are many vertices of $G\in \RT$ in the symmetric difference of $\Bb(A,r)$ and $\Bb(B,r)$ provided that $A$ and $B$ are ``sufficiently far from each other''.

\begin{lemma}\label{lem:vertices_in_crown}
Fix $r = r(n) \in [r_0, \sqrt{n}/4)$ and let $G\in \RT$.
Let 
\begin{equation*}
\eps_c = 
\begin{cases}
12 \, r^{-1/3}, & \text{ if } r \ge \log^{3/2} n, \text{ and}\\
12 \log n / r, & \text{ otherwise.}
\end{cases}
\end{equation*}
Then a.a.s. the following property holds: for any pair of vertices of $G$ with positions $A,B$ such that $d_T(A,B)=\eps \ge \eps_c$, the number of vertices in $\Bb(A,r)\triangle \Bb(B,r)$ is at least $\min(\eps, 2r) r$. 
\end{lemma}
\begin{proof}
Consider the positions $A,B$ of \emph{any} two vertices of $G$ at distance $\eps \ge \eps_c$ from each other. 
By Observation~\ref{ob AB}, $| \Bb(A,r)\triangle \Bb(B,r) | \ge a:= 2 \min(\eps, 2r) r$. Hence, the number of vertices in the symmetric difference can be stochastically bounded from below by a random variable $X \sim \textrm{Bin}(n-2, a/n)$ with $\E X = (n-2)a/n = (1+o(1)) a$. Note that if $r \ge \log^{3/2} n$, then $a \ge 2 \eps_c r = 24 r^{2/3} \ge 24 \log n$; otherwise, since $\varepsilon_c \le 12 \log n/ r_0 \le 2r_0\le 2r$, it is also true that $a \ge 2 \eps_c r = 24 \log n$.
In either case, it follows from the Chernoff's bound~(\ref{chern}), applied with $t = \E X - a/2 = (1+o(1)) a/2$, that 
\begin{eqnarray*}
\Prob(X \le a/2) &=& \Prob( X \le \E X - t ) \le \exp \left( - \frac {t^2}{2 \E X} \right) = \exp \left( - \left( \frac {1}{8}+o(1) \right) a \right) \\
&\le& \exp \left( - (3+o(1)) \log n \right) = o \left( \frac {1}{n^2} \right).
\end{eqnarray*}
The lemma holds by a union bound over all pairs of vertices.
\end{proof}

The next lemma controls the number of pairs of vertices at a given distance from each other.

\begin{lemma}\label{lem:pairs_at_given_distance}
Fix $r = r(n) \in [r_0, \sqrt{n}/4)$ and let $G\in \RT$. Then, a.a.s.\ the following properties hold for all squares $S \in \mathcal{F}$.
\begin{itemize}
\item [(a)] The number of vertices in $S$ is at most $2 \cdot 10^{10} r^2$. 
\item [(b)] Let $\eps_c = 12\, r^{-1/3}$. If $r \ge \log^{3/2} n$, then for any given $k = k(n) \in \N \cup \{0\}$ satisfying $\eps = \eps(k) := 2^k \varepsilon_c \le r^{-0.1}$, the number of pairs of vertices in $I(S)$ that are at distance at most $\eps$ from each other is at most $2 \cdot 10^{12} r^2 \eps^2$.
\item [(c)] Let $\eps_c = 12\, \log n / r$. If $\log^{5/4} n \le r < \log^{3/2} n$, then for any given $k = k(n) \in \N \cup \{0\}$ satisfying $\eps = \eps(k) := 2^k \varepsilon_c \le r^{-0.1}$, the number of pairs of vertices in $I(S)$ that are at distance at most $\eps$ from each other is at most $2 \cdot 10^{12} r^2 \eps^2$.
\item [(d)] Let $\eps_c = 12\, \log n / r$. If $100 \log n \le r < \log^{5/4} n$, then the number of pairs of vertices in $I(S)$ that are at distance at most $\eps_c$ from each other is at most $10^{16} \log^2 n$.
\end{itemize}
\end{lemma}
\begin{proof}
We prove part (a) first. Let us concentrate on any square $S \in \mathcal{F}$. Recall that the area of $S$ is $(10^5 r)^2 = 10^{10} r^2$. Hence, the number of vertices in $S$ is equal to $X \sim \textrm{Bin}(n, 10^{10} r^2/n)$ with $\E X = 10^{10} r^2 \ge 49 \cdot 10^{12} \log n$. It follows immediately from Chernoff's bound \eqref{chern1} that 
\begin{eqnarray*}
\Prob( X \ge 2 \cdot 10^{10} r^2 ) &=& \Prob( X \ge \E X + \E X) \le \exp \left( - \frac {(\E X)^2}{2 ( \E X + \E X/3 )} \right) \\
&=& \exp \left( - \frac 38 \, \E X \right) \le \exp \left( - 10^{12} \log n \right) = o(1/n).
\end{eqnarray*}
Since the number of squares in $\mathcal{F}$ is less than $n$, the desired conclusion holds by a union bound over all squares. 

\bigskip

In order to simplify the argument for part (b), we will use the de-Poissonization technique mentioned in Section~\ref{sec:Poisson}. As before, let us concentrate on any square $S \in \mathcal{F}$. Without loss of generality, we may assume that the left-bottom corner of $S$ is the point $(0,0)$ and the right-top corner of $S$ is the point $(10^5 r, 10^5 r)$. Let us also fix $k = k(n) \in \N \cup \{0\}$ satisfying $\eps = \eps(k) := 2^k \varepsilon_c \le r^{-0.1}$. 

For a given $a,b \in \{0,1\}$, let $\mathcal{E}_{a,b}$ be a family of \emph{small} squares of side length $2 \eps$, with sides parallel to the axes, and with left-bottom vertices at points $((a+2i) \eps, (b+2j)\eps)$ for some $i,j \in \N \cup \{0\}$ such that $(a+2i) \eps < 10^5 r$ and $(b+2j)\eps < 10^5 r$. Clearly, $|\mathcal{E}_{a,b}| = (1+o(1)) 10^{10} r^2/4\eps^2 = O( (r/\eps_c)^2) = O( r^{8/3} ) = O(n^{4/3})$. Moreover, any pair of vertices in $I(S)$ that are at distance at most $\eps$ from each other has to be included in some small square in some of the four families $\mathcal{E}_{a,b}$. Hence it is enough to bound from above the number of pairs of vertices contained in a small square in any of the four families $\mathcal{E}_{a,b}$.

Let us concentrate on one family $\mathcal{E}_{a,b}$ for a given $a,b \in \{0,1\}$. For any small square $s \in \mathcal{E}_{a,b}$, the number of vertices in $s$ is equal to $\mathrm{Po}(\lambda)$ with $\lambda = (2\eps)^2 = 4\eps^2$. 
For $\ell \ge 2$, let $Z_s^{\ge \ell}$ be the random variable counting the number of vertices in $s$ if this number is at least $\ell$, and $0$ otherwise. For $k \ge \ell$, we have 
$$
\Prob \left( Z_s^{\ge \ell} = k \right) = \Prob \left( \mathrm{Po}(\lambda) = k \right) = \frac {\lambda^k}{k!} \exp(-\lambda),
$$
whereas for $0 < k < \ell$ the probability is $0$ by definition.
Since $\eps \le r^{-0.1} = o(1)$ (and so also $\lambda=o(1)$), 
$$
\E Z_s^{\ge \ell} = \sum_{k \ge \ell} k \cdot \frac {\lambda^k}{k!} \exp(-\lambda) = (1+o(1)) \frac {\lambda^{\ell}}{(\ell-1)!} \exp(-\lambda) = (1+o(1)) \frac {(4\eps^2)^{\ell}}{(\ell-1)!}.
$$
For every fixed $\ell \ge 2$, since the random variables $(Z_s^{\ge \ell})_{s\in \mathcal{E}_{a,b}}$ are independent, we may apply Lemma~\ref{bentkus} with $X_s=Z_s^{\ge \ell}$ and $S= S^{\ge \ell}=\sum_{s \in \mathcal{E}_{a,b}} Z_s^{\ge \ell}$. In other words, $S^{\ge \ell}$ counts the number of vertices in small squares containing at least $\ell$ vertices. Thus, for every fixed $\ell\ge 2$ we have $\E Z_s^{\ge \ell}=O(\eps^{2\ell}) = o(1)$ and 
\begin{eqnarray}
\E S^{\ge \ell} &=& (1+o(1))|\mathcal{E}_{a,b}|(4\eps^2)^{\ell}/(\ell-1)!  = (1+o(1)) \frac {10^{10} r^2 (4\eps^2)^{\ell-1}}{(\ell-1)!}. \label{eq:expectation}
\end{eqnarray}
Since for every $s\in \mathcal E_{a,b}$ we have $\E Z_s^{\ge \ell} \le 1$, we may fix $m = 1$. Note also that $Z_s^{\ge \ell}$ attains no value between $0$ and $1$. Thus, we may simply choose (using the notation introduced right before Lemma~\ref{bentkus}) $\mathcal{L}(Z_s^{\ge \ell})=\mathcal{L}(Y^{[m]})$ and $T = \sum_{s\in \mathcal E_{a,b}} Z_s^{\ge \ell}$; in particular, $\E e^T=\E e^{S^{\ge \ell}}$. Since the random variables $(Z_s^{\ge \ell})_{s\in \mathcal{E}_{a,b}}$ are independent, we have 
\begin{eqnarray*}
\E e^{S^{\ge \ell}} &=& \left( \E e^{Z_s^{\ge \ell}} \right)^{|\mathcal E_{a,b}|}= \left( \sum_{k=0}^{\ell-1} \lambda^k e^{-\lambda}/k! + \sum_{k \ge \ell} (e \lambda)^k e^{-\lambda}/k! \right)^{|\mathcal E_{a,b}|} \le \left( 1 + (1+o(1)) \frac{(4e \eps^2)^{\ell}}{\ell !} \right)^{|\mathcal E_{a,b}|} \\
&\le& \exp \left( (1+o(1)) \frac{(4e \eps^2)^{\ell}}{\ell!} |\mathcal E_{a,b}| \right). 
\end{eqnarray*}
By Lemma~\ref{bentkus} applied with $S=S^{\ge \ell}$, we have
\begin{eqnarray*}
\Prob \left( S^{\ge \ell} \ge e^{\ell} \E S^{\ge \ell} \right) &\le&  e^{-e^{\ell} \E S^{\ge \ell}} \E e^T = \exp \left( - e^{\ell} (1+o(1))\frac{(4\eps^2)^{\ell}\ell}{\ell !}|\mathcal{E}_{a,b}| + (1+o(1)) \frac{(4e \eps^2)^{\ell}}{\ell !}|\mathcal{E}_{a,b}| \right) \nonumber \\
&=& \exp \left( - (1+o(1)) \frac {\ell-1}{\ell} \ e^{\ell} \ \E S^{\ge \ell} \right) \le \exp \Big( - (e^2/2+o(1)) \E S^{\ge \ell} \Big). \nonumber 
\end{eqnarray*}
Hence, as long as $\E S^{\ge \ell} \ge \log n$, 
\begin{equation}
    \Prob \left( S^{\ge \ell} \ge e^{\ell} \E S^{\ge \ell} \right) = o(1/n^2).\label{bentkusapplied}
\end{equation}

Recall that by~\eqref{eq:expectation} for every $\ell\ge 2$ we have $\E S^{\ge \ell} = \Theta( r^2 \eps^{2 \ell - 2} )$. In particular, $\E S^{\ge 2} = \Theta( r^2 \eps^2 ) = \Omega( r^2 \eps_c^2 ) = \Omega( r^{4/3} ) = \Omega ( \log^2 n )$. Since $\eps_c \le \eps \le r^{-0.1}$, there exists some integer $\ell_0\in [3,11]$ such that $\E S^{\ge j} \ge \log n$ for every integer $j\in [2, \ell_0-1]$, whereas $\E S^{\ge \ell_0} < \log n$. Therefore, it follows from~\eqref{bentkusapplied} that for every integer $j\in [2, \ell_0-1]$, with probability $1-o(1/n^2)$ the value of $S^{\ge j}$ is at most a constant multiplicative factor away from its expectation. 
Observe that if the number of vertices in one small square is $j$, where $2 \le j < \ell_0$, then trivially each vertex belongs to exactly $j-1$ pairs of vertices from this square. We get that with probability $1-o(1/n^2)$ the number of pairs involving such vertices is at most
\begin{equation}\label{small j}
\sum_{j = 2}^{\ell_0-1} (j-1) S^{\ge j} = (1+o(1)) S^{\ge 2} \le (4 e^2 \cdot 10^{10}+o(1)) r^2 \eps^2.   
\end{equation}

On the other hand, one may couple the variables $Z^{\ge \ell_0}_s$ with variables $\hat{Z}^{\ge \ell_0}_s$ in such a way that $Z^{\ge \ell_0}_s \preccurlyeq \hat{Z}^{\ge \ell_0}_s$, and such that $\E \hat{S}^{\ge  \ell_0} = \log n$, where $\hat{S}^{\ge \ell_0}:=\sum_{s \in \mathcal{E}_{a,b}} \hat{Z}^{\ge \ell_0}_s$. More precisely, we set up the coupling such that for all $k \ge \ell_0$ we have
$$
\Prob \left( \hat{Z}^{\ge \ell_0}_s = k \right) = \frac {\hat{\lambda}^k}{k!} e^{-\hat{\lambda}},
$$ 
where $\hat{\lambda}=4\hat{\eps}^2$ for some carefully tuned value of $\hat{\eps} \ge \eps$ such that $\E \hat{S}^{\ge \ell_0}=\log n$. (Similarly to the original random variable $Z^{\ge \ell_0}_s$, $\hat{Z}^{\ge \ell_0}_s$ attains no value smaller than $\ell_0$ other than $0$.) Clearly, $S^{\ge \ell_0} \le \hat{S}^{\ge \ell_0}$ and $\E \, e^{S^{\ge \ell_0}} \le \E e^{\hat{S}^{\ge \ell_0}}$. 
We may apply Lemma~\ref{bentkus} again, this time with $T=S=\hat{S}^{\ge  \ell_0}$. Arguing as in~\eqref{bentkusapplied}, we get that
$$
\Prob \left( S^{\ge \ell_0} \ge e^{\ell_0} \log n \right) \le
\Prob \left( \hat{S}^{\ge \ell_0} \ge e^{\ell_0} \E \hat{S}^{\ge \ell_0} \right) \le
\exp \Big( -(e^2/2+o(1)) \E \hat{S}^{\ge \ell_0} \Big) = o(1/n^2).
$$
We deduce that with probability $1-o(1/n^2)$ there are at most $\binom{e^{\ell_0} \log n}{2} \le \binom{e^{11} \log n}{2} \le 10^{10} \log^2 n$ pairs of vertices such that each pair belongs to some small square containing at least $\ell_0$ vertices. Combining this observation with~\eqref{small j} we get that with probability $1 - o(1/n^2)$, the number of pairs of vertices that are both contained in one square in the family $\mathcal{E}_{a,b}$ is at most 
$$
(4 e^2 \cdot 10^{10}+o(1)) r^2 \eps^2 + 10^{10} \log^2 n \le 5 \cdot 10^{11} r^2 \eps^2.
$$
(Note that $(4 e^2 \cdot 10^{10}+o(1)) r^2 \eps^2 \ge 4320 \cdot 10^{10} \log^2 n$ and so the second term is much smaller than the first one.)

Taking a union bound over all four families $\mathcal{E}_{a,b}$, all $O(n)$ squares $S$, and all $O(\log n)$ values of $k$, we get that the desired bound holds for the Poisson model with probability $1-o(1/\sqrt{n})$, and so it holds a.a.s.\ for $\RT$.

\bigskip

Parts~(c) and~(d) are similar to part~(b) so we only sketch the proof highlighting a few minor adjustments to the argument. In fact, part~(c) follows \emph{exactly} the same argument, since $\eps \le r^{-0.1} = o(1)$ as before. The only thing that is worth pointing out is that the new definition of $\eps_c$, namely, $\eps_c = 12\, \log n / r$ guarantees that $\E S^{\ge 2} = \Theta( r^2 \eps^2 ) = \Omega( r^2 \eps_c^2 ) = \Omega ( \log^2 n )$, as needed.

\bigskip

Part~(d) requires slightly more careful adjustments since $\eps_c$ might \emph{not} tend to zero as $n\to \infty$. As before, the number of vertices in $s \in \mathcal{E}_{a,b}$ is equal to $\mathrm{Po}(\lambda)$, but this time $\lambda = (2\eps_c)^2 = 4\eps_c^2 \le 1/10$ since $r \ge 100 \log n$. We keep the same notation: for $\ell \ge 2$, let $Z_s^{\ge \ell}$ be the random variable counting the number of vertices in $s$, if this number is at least $\ell$, and $0$ otherwise. This time we get 
$$
\E Z_s^{\ge \ell} = \sum_{k \ge \ell} k \cdot \frac {\lambda^k}{k!} \exp(-\lambda) = \frac {C_\ell \cdot \lambda^{\ell}}{(\ell-1)!} \exp(-\lambda), \text{ where } C_{\ell} := \sum_{k \ge \ell} \frac { \lambda^{k-\ell} (\ell-1)!}{(k-1)!}.
$$
Note that $C_{\ell}$ is an explicit constant between 1 and 2 as each term in the sum is at most half of the previous term. It follows that
\begin{eqnarray*}
\E S^{\ge \ell} &=& C_{\ell} \cdot  |\mathcal{E}_{a,b}| \frac {\lambda^{\ell}}{(\ell-1)!} e^{-\lambda} = (1+o(1)) \frac {10^{10} C_{\ell} \, r^2 (4\eps_c^2)^{\ell-1}}{(\ell-1)!} e^{-4 \eps_c^2} \\
&\le& (1+o(1)) \frac {8 \cdot 10^{10} (r \eps_c)^2}{(\ell-1)!} \le \dfrac{144 \cdot 10^{11} \log^2 n}{(\ell-1)!}.
\end{eqnarray*}

In particular, $\E S^{\ge 2} = \Theta(\log^2 n)$. Hence, there exists $\ell_0$ such that $\frac {j-1}{j} e^{j} \, \E S^{\ge j} \ge 3 \log n$ for every integer $j\in [2, \ell_0-1]$, whereas $\frac {\ell_0-1}{\ell_0} e^{\ell_0} \, \E S^{\ge \ell_0} < 3 \log n$. Arguing as before (including the coupling that is needed for the claim for $\ell_0$), we get that with probability $1-o(\ell_0/n^2) = 1-o(\log n / n^2)$, we have $S^{\ge j} \le e^{j} \E S^{\ge j}$ for $j < \ell_0$, and $S^{\ge \ell_0} \le e^{\ell_0} \E S^{\ge \ell_0} < \frac {\ell_0}{\ell_0-1} \cdot 3 \log n \le 6 \log n$. 
With probability $1-o(\log n/n^2)$, the number of pairs of vertices that are both contained in one of the small squares in the family $\mathcal{E}_{a,b}$ is at most 
$$
\sum_{j = 2}^{\ell_0-1} (j-1) S^{\ge j} + \binom{S^{\ge \ell_0}}{2}\le \sum_{j=2}^{\ell_0-1} \dfrac{144 \cdot 10^{11} e^j \log^2 n}{(j-2)!} + 18\log^2 n\le 2\cdot 10^{15} \log^2 n. 
$$
The claim may now be deduced after the union bound over the four families $(\mathcal E_{a,b})_{a,b\in \{0,1\}}$.
\end{proof}

Now, we are ready to prove Lemma~\ref{lem:second_phase} and finish the proof of the upper bounds. 

\begin{proof}[Proof of Lemma~\ref{lem:second_phase}]
Since we aim for a statement that holds a.a.s., we may assume that the properties stated in Lemma~\ref{lem:vertices_in_crown} and Lemma~\ref{lem:pairs_at_given_distance} hold. In other words, we do not generate a random graph from $\RT$ but instead consider a deterministic graph $G$ that satisfies the desired properties. Let $S \in \mathcal{F}$. We will use a non-constructive argument to show that there exists a set $W = X \cup Y \subseteq S$ such that all vertices in $I(S)$ have a unique $W$-signature.

\bigskip

Let us first concentrate on dense graphs and assume that $r \ge \log^{3/2} n$. We construct a random set $X$ by independently selecting vertices from $S$ to be put into $X$ with probability $\delta = r^{-2/3}$. (Note that this is the only source of randomness at this point as $G$ is a deterministic graph.) By Lemma~\ref{lem:pairs_at_given_distance}(a), the number of vertices in $S$ is at most $2\cdot 10^{10} r^2$, and so $\E |X| \le 2\cdot 10^{10} \, r^{4/3}$.

By Lemma~\ref{lem:pairs_at_given_distance}(b), there are at most $2 \cdot 10^{12} r^2 \eps_c^2 = 288 \cdot 10^{12} \, r^{4/3}$ pairs of vertices in $I(S)$ at distance at most $\eps_c$ from each other. The number of these is small enough so we do not need to worry about them; all vertices involved in such pairs may simply be put into $Y$. Fix any $k = k(n) \in \N$ such that $2^k \eps_c \le r^{-0.1} = o(1)$. Concentrate now on any pair of vertices $u, v$ from $I(S)$ that are at distance $\eps$ from each other for some $2^{k-1} \eps_c < \eps \le 2^k \eps_c$. By Lemma~\ref{lem:pairs_at_given_distance}(b), there are at most $2 \cdot 10^{12} r^2 (2^k \eps_c)^2 = 288 \cdot 10^{12} \, r^{4/3} \cdot 4^k$ such pairs. By Lemma~\ref{lem:vertices_in_crown}, there are at least $\eps r$ vertices in $\Bb(u,r)\triangle \Bb(v,r)$. Since this symmetric difference is included in $S$, each of these vertices independently ends up in $X$ with probability $\delta$. If at least one of them actually ends up in $X$, then the vertices $u,v$ are distinguished by the sensors. Hence, $u,v$ are \emph{not} distinguished with probability at most 
\begin{equation*}
\left(1 - \delta \right)^{\eps r} \le \exp \left( - \delta \eps r \right) \le \exp \left( - 2^{k-1} \delta \eps_c r \right) = \exp \left( - 6 \cdot 2^k \right).    
\end{equation*}
Similarly, by Lemma~\ref{lem:pairs_at_given_distance} (a), there are trivially at most $O((r^2)^2)= O(r^4)$ pairs of vertices in $I(S)$ at distance at least $\eps := r^{-0.1}$ from each other. By Lemma~\ref{lem:vertices_in_crown}, there are at least $\eps r = r^{0.9}$ vertices in $\Bb(u,r)\triangle \Bb(v,r)$ for any such pair of vertices $u,v$, and so they are \emph{not} distinguished with probability at most
$$
(1 - \delta)^{r^{0.9}} \le \exp ( - \delta r^{0.9} ) \le \exp ( - r^{0.2} ).
$$
Combining all of these observations together we get that the expected number of pairs of vertices with the same $X$-signature is at most 
\begin{align*}
288 \cdot 10^{12} \, r^{4/3} & + \sum_{k \ge 1} 288 \cdot 10^{12} \, r^{4/3} \cdot 4^k \cdot \exp \left( - 6 \cdot 2^k \right) + O(r^4) \cdot \exp( - r^{0.2} ) \\
&\le 288 \cdot 10^{12} \, r^{4/3} + 0.01 \cdot 288 \cdot 10^{12} \, r^{4/3} + o(1) \le 300 \cdot 10^{12} \, r^{4/3}.
\end{align*}
As promised, we put all vertices that occur in at least one such pair into the set $Y$. 

Clearly, by construction each vertex in $I(S)$ has a unique $W$-signature. Moreover, we get that $\E |W| = \E |X| + \E |Y| \le  2\cdot 10^{10} \, r^{4/3} + 6 \cdot 10^{14} \, r^{4/3} \le 10^{15} \, r^{4/3}$. Finally, the probabilistic method implies that there exists a set $W$ of size at most $10^{15} \, r^{4/3}$ and the proof for the dense graphs is finished.

\medskip

Let us now deal with sparser graphs and assume that $100 \log n \le r < \log^{3/2} n$. The proof only requires small adjustments so we only sketch it. We construct a random set $X$ by independently selecting vertices from $S$ to be put into $X$ with probability $\delta = \log^2 n / r^2$ and so, by Lemma~\ref{lem:pairs_at_given_distance}(a), $\E |X| \le 2\cdot 10^{10} \, \log^2 n$. 

Suppose first that $\log^{5/4} n \le r < \log^{3/2} n$. By Lemma~\ref{lem:pairs_at_given_distance}(c), there are at most $2 \cdot 10^{12} r^2 \eps_c^2 = 288 \cdot 10^{12} \log^2 n$ pairs of vertices in $I(S)$ at distance at most $\eps_c = 12 \log n / r$ from each other. Fix any $k = k(n) \in \N$ such that $2^k \eps_c \le r^{-0.1} = o(1)$, and concentrate on any pair of vertices $u, v$ from $I(S)$ that are at distance $\eps$ from each other for some $2^{k-1} \eps_c < \eps \le 2^k \eps_c$. This pair of vertices is \emph{not} distinguished with probability at most 
\begin{equation*}
\left(1 - \delta \right)^{\eps r} \le \exp \left( - 2^{k-1} \delta \eps_c r \right) = \exp \left( - 6 \cdot 2^k \log^3 n / r^2 \right) \le \exp \left( - 6 \cdot 2^k \right).
\end{equation*}
On the other hand, pairs of vertices that are at distance at least $\eps:=r^{-0.1}$ are \emph{not} distinguished with probability at most $(1-\delta)^{r^{0.9}} \le \exp( - \log^2 n / r^{1.1} ) \le \exp( - \log^{0.35} n)$. It follows that
\begin{eqnarray*}
\E |W| &=& \E |X| + \E |Y| \le 2\cdot 10^{10} \, \log^2 n \\
&& + 2 \left( 288 \cdot 10^{12} \log^2 n + \sum_{k \ge 1} 288 \cdot 10^{12} \log^2 n \cdot 4^k \cdot \exp \left( - 6 \cdot 2^k \right) + O(r^4) \cdot \exp(- \log^{0.35} n) \right) \\
&\le& 10^{15} \log^2 n.
\end{eqnarray*}

Suppose then that $100 \log n \le r < \log^{5/4} n$. By Lemma~\ref{lem:pairs_at_given_distance}(d), there are at most $10^{16} \log^2 n$ pairs of vertices in $I(S)$ at distance at most $\eps_c = 12 \log n / r$ from each other. The remaining pairs of vertices are \emph{not} distinguished with probability at most 
\begin{equation*}
\left(1 - \delta \right)^{\eps_c r} \le \exp \left( - \delta \eps_c r \right) = \exp \left( - 12 \log^3 n / r^2 \right) \le \exp \left( - 12 \log^{1/2} n \right).
\end{equation*}
This time
\begin{eqnarray*}
\E |W| &\le& (2\cdot 10^{10} \, \log^2 n) + 2 \left( 10^{16} \log^2 n + O(r^4) \cdot \exp(- 12 \log^{1/2} n) \right) \\
&\le& 3 \cdot 10^{16} \log^2 n.
\end{eqnarray*}

\bigskip

The upper bound for very sparse graphs is trivial. If $r < 100 \log n$, then one may simply put sensors on all vertices in $S$, that is, take $\delta = 1$. The bound follows immediately from Lemma~\ref{lem:pairs_at_given_distance}(a).
\end{proof}

\section{Lower bound}

This section is devoted to the proof of the lower bounds stated in Theorem~\ref{thm:main}. Assume first that $r = r(n) \ge \log n$; we will adjust the argument to sparser graphs at the end of this section.  Let $\Bb_R$ be the ball with radius $r/3$, centered in the center $O$ of the square $[0, \sqrt{n}]^2$. We will show that if the number of sensors is less than the lower bound given by Theorem~\ref{main: lower bound}, a.a.s.\ the robber has a strategy to remain undetected forever while staying in the ball $\Bb_R$ during the entire game.

\begin{theorem}\label{main: lower bound}
Fix $r = r(n) \ge \log n$ and let $G\in \RT$. Then, a.a.s.\ the robber may remain undetected forever in $\Bb_R$ in the presence of less than $10^{-4} \, r^{4/3} / \log^{1/3} n$ sensors at each round.
\end{theorem}

The general idea behind the proof of the lower bound is quite natural and intuitive. First, for a carefully tuned function $\eps=\eps(r)$, we will show that there are relatively many pairs of vertices in $\Bb_R$ that are at distance at most $\eps$ from each other. In fact, in order to simplify the argument we will concentrate on a particular \emph{special} sub-family of pairs of such vertices, which we will call \emph{special pairs}, that satisfy some additional useful property. On the other hand, we will show that regardless of where a single sensor is placed, it distinguishes only a few pairs of such vertices. This will immediately imply the desired lower bound for the number of sensors needed to distinguish vertices in $\Bb_R$ and so to locate the robber hiding in $\Bb_R$. 

A family of pairs of vertices from $\Bb_R$ that are at distance at most $\eps$ from each other is called \emph{$\eps$-special} if each vertex in $\Bb_R$ belongs to at most one such pair. In other words, an $\eps$-special family induces a matching. We will start by showing that there exists a large $\eps$-special family, provided that the graph is dense enough.

\begin{lemma}\label{lem:special_family}
Fix $r = r(n) \ge \log n$ and let $G\in \RT$. Fix $\eps = \eps(n) = (\log n / r)^{1/3} \le 1$. Then, a.a.s.\ there exists an $\eps$-special family of pairs of vertices of size $r^2 \eps^2 / 100$.
\end{lemma}
\begin{proof}
It will be convenient to use the de-Poissonization technique as explained in Section~\ref{sec:Poisson}. We start with tessellating the entire torus into squares of side length $\eps/\sqrt{2}$. Trivially, any two vertices that belong to the same square are at distance at most $\eps$. Since the area of each square (namely, $\eps^2 / 2 \le 1/2$) is negligible in comparison to the area of the ball $\Bb_R$ (namely, $\pi r^2 / 9 \ge \pi \log^2 n / 9$), the number of squares that are completely inside the ball is equal to $\ell = (2\pi/9+o(1)) (r^2 / \eps^2)$.

We construct a special family of pairs as follows. We independently expose the vertices in each square that is completely inside the ball, and if exactly two vertices belong to a given square, then we add this pair to the family. The probability that a given square has exactly two vertices in it is equal to 
$$
p := \frac {(\eps^2/2)^2}{2!} \exp( - \eps^2/2) \ge \frac {\eps^4}{8 e^{1/2}}.
$$
Hence, the number of special pairs in $\Bb_R$ is stochastically bounded from below by the random variable $X \sim \textrm{Bin}(\ell,p)$ with $\E X = \ell p \ge r^2 \eps^2 / 36 = \Theta( (\log n)^{2/3} r^{4/3} ) = \Omega( \log^2 n)$. It follows immediately from Chernoff's bound~\eqref{chern} that with probability $1-o(1/\sqrt{n})$ the size of our $\eps$-special family is at least $\E X/2 \ge r^2 \eps^2 / 100$. By de-Poissonization the same property holds a.a.s.\ in $\RT$ and so the proof of the lemma is finished.
\end{proof}

Suppose that a sensor is placed on a vertex $v \in V$ of a connected geometric graph. For a given non-negative integer $i$, let $D_{i}(v)$ be the set of vertices that are at graph distance $i$ from $v$ in $G$. Since vertices in $\Bb_R$ induce a complete graph, putting a sensor on $v$ divides $\Bb_R \cap V$ into the set of vertices in $\Bb_R \cap D_k(v)$ (possibly empty) at distance $k$ from $v$ and the set of vertices in $\Bb_R \cap D_{k+1}(v)$ (again, possibly empty) that are at distance $k+1$ from $v$, where $k \in \N \cup \{0\}$: indeed, if a vertex $u$ in $\Bb_R$ is at graph distance $k$ from $v$, then every other vertex in $\Bb_R$ is at graph distance at most $k+1$ from $v$. Note that, in particular, if $v \in \Bb_R$, then $D_0(v) = \{v\}$ and $D_1(v)$ contains all other vertices in $\Bb_R$, so this sensor only distinguishes itself from the remaining vertices in $\Bb_R$. The partition of $\Bb_R \cap V$ is more challenging to investigate when the sensor is placed on a vertex $v \in V \setminus \Bb_R$ so that $k \neq 0$. Let us concentrate on this situation.

Note that all vertices in $D_k(v)$ belong to 
$$
U(D_{k-1}(v)) := \bigcup_{u \in D_{k-1}(v)} \Bb(u,r).
$$ 
The argument that will be used in the proof of Theorem~\ref{main: lower bound} will show that $U(D_{k-1}(v))$ has a non-empty intersection with $\Bb_R$. On the other hand, no vertex in $D_{k+1}(v)$ belongs to $U(D_{k-1}(v))$. Hence, in order to estimate the number of $\eps$-special pairs of vertices that are distinguished by $v$ we need to concentrate on the \emph{boundary} between the sets $\Bb_R \cap U(D_{k-1}(v))$ and $\Bb_R \setminus U(D_{k-1}(v))$. Note also that every vertex in an $\eps$-special pair that is distinguished by $v$ must be at Euclidean distance at most $\eps$ from the boundary of $\Bb_R \cap U(D_{k-1}(v))$ in $\Bb_R$, a key property that will be used to estimate the number of such pairs. 

\bigskip

In order to investigate the ``shape'' of the boundary, we relax the assumption about $D_{k-1}(v)$ (vertices at distance $k-1$ from $v$) and simply assume that it is \emph{any} finite set of points $\{ O_i \}_{i\in \mathcal I}$ with the property that $U(\{ O_i \}_{i\in \mathcal I})$ has nonempty intersection with $\Bb_R$. Note that there could be many disconnected boundary regions of $\Bb_R\setminus U(\{ O_i \}_{i\in \mathcal I})$, and it will be convenient to distinguish between small and large regions. For the latter family of regions, we could use Weyl's famous tube formula (see Chapter 17.2 in~\cite{Mor}), but we do not do that for the following two reasons. First of all, we decided to provide an independent proof to keep the presentation self-contained. Moreover, the application of Weyl's formula would require introducing certain curvature concepts from differential geometry, which would make the argument comparable in length to our elementary proof but slightly more technical. 

\bigskip

Let us now start with a few preparatory geometric lemmas.

\begin{lemma}\label{lem length}
Let $k_1$ and $k_2$ be two circles with centers $O_1$ and $O_2$ and radii $r_1\le r_2$, respectively. Suppose that $k_1\cap k_2 = \{A, B\}$ and let $\ell_1, \ell_2$ be two lines, parallel to $O_1O_2$, that divide the plane in three parts such that $A, O_1$ and $B$ are all in different parts. Let $\ell_1\cap k_1 = \{P_1, S_1\}$, $\ell_2\cap k_1 = \{Q_1, R_1\}$, $\ell_1\cap k_2 = \{P_2, S_2\}$ and $\ell_2\cap k_2 = \{Q_2, R_2\}$ such that $P_1, Q_1$ are on the same side of the line $AB$ and also $P_2, Q_2$ are on the same side of $AB$. Then, $|P_1Q_1|\ge |P_2Q_2|$.
\end{lemma}

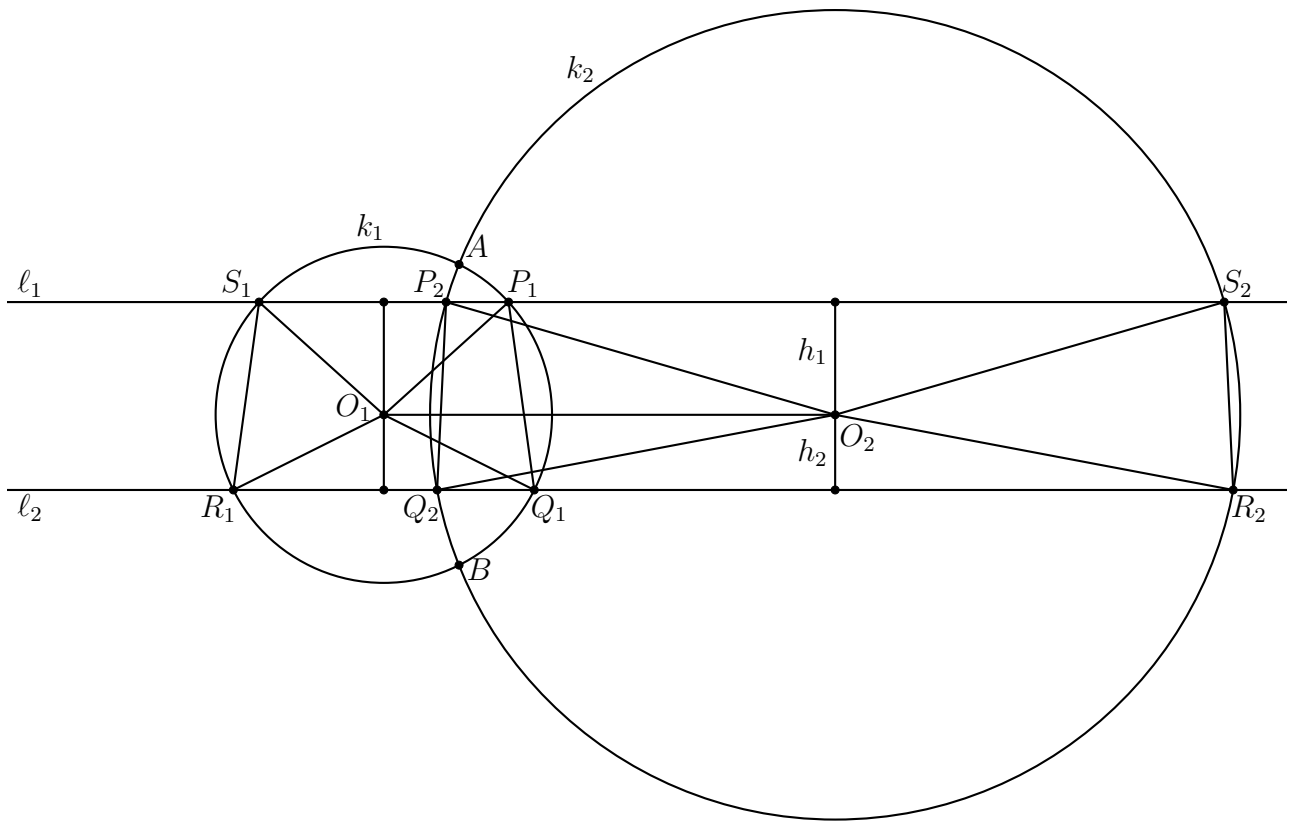
\begin{figure}
\centering
\begin{tikzpicture}[line cap=round,line join=round,x=1cm,y=1cm]
\clip(-7,-5.5) rectangle (10,5.5);
\draw [line width=0.8pt] (-2,0) circle (2.23606797749979cm);
\draw [line width=0.8pt] (4,0) circle (5.385164807134505cm);
\draw [line width=0.8pt,domain=-9.693690453732609:14.784948492431132] plot(\x,{(-4-0*\x)/4});
\draw [line width=0.8pt,domain=-9.693690453732609:14.784948492431132] plot(\x,{(--6-0*\x)/4});
\draw [line width=0.8pt] (4,1.5)-- (4,-1);
\draw [line width=0.8pt] (-0.3416876048222998,1.5)-- (0,-1);
\draw [line width=0.8pt] (-1.172040216394301,1.5)-- (-1.2915026221291819,-1);
\draw [line width=0.8pt] (9.1720402163943,1.5)-- (9.291502622129181,-1);
\draw [line width=0.8pt] (-3.6583123951777,1.5)-- (-4,-1);
\draw [line width=0.8pt] (-3.6583123951777,1.5)-- (-2,0);
\draw [line width=0.8pt] (-2,0)-- (-4,-1);
\draw [line width=0.8pt] (4,0)-- (9.1720402163943,1.5);
\draw [line width=0.8pt] (4,0)-- (9.291502622129181,-1);
\draw [line width=0.8pt] (-2,1.5)-- (-2,-1);
\draw [line width=0.8pt] (-2,0)-- (-0.3416876048222998,1.5);
\draw [line width=0.8pt] (-1.172040216394301,1.5)-- (4,0);
\draw [line width=0.8pt] (-2,0)-- (4,0);
\draw [line width=0.8pt] (-1.2915026221291819,-1)-- (4,0);
\draw [line width=0.8pt] (-2,0)-- (0,-1);
\begin{scriptsize}
\draw [fill=black] (-2,0) circle (1.5pt);
\draw[color=black] (-2.4,0.1) node {\large{$O_1$}};
\draw [fill=black] (-1,2) circle (1.5pt);
\draw[color=black] (-0.7628120257181813,2.247940051813796) node {\large{$A$}};
\draw[color=black] (-2.1779833397932724,2.5) node {\large{$k_1$}};
\draw [fill=black] (4,0) circle (1.5pt);
\draw[color=black] (4.3,-0.3) node {\large{$O_2$}};

\draw[color=black] (3.7,-0.5) node {\large{$h_2$}};
\draw[color=black] (3.7,0.85) node {\large{$h_1$}};

\draw[color=black] (0.6141114150035292,4.6) node {\large{$k_2$}};
\draw[color=black] (-6.7,-1.23) node {\large{$\ell_2$}};
\draw[color=black] (-6.7,1.75) node {\large{$\ell_1$}};
\draw [fill=black] (-1,-2) circle (1.5pt);
\draw[color=black] (-0.73,-2.05) node {\large{$B$}};
\draw [fill=black] (-1.172040216394301,1.5) circle (1.5pt);
\draw[color=black] (-1.4,1.75) node {\large{$P_2$}};
\draw [fill=black] (-1.2915026221291819,-1) circle (1.5pt);
\draw[color=black] (-1.5,-1.25) node {\large{$Q_2$}};
\draw [fill=black] (9.291502622129181,-1) circle (1.5pt);
\draw[color=black] (9.5,-1.25) node {\large{$R_2$}};
\draw [fill=black] (4,1.5) circle (1.5pt);
\draw [fill=black] (4,-1) circle (1.5pt);
\draw [fill=black] (9.1720402163943,1.5) circle (1.5pt);
\draw[color=black] (9.34,1.75) node {\large{$S_2$}};
\draw [fill=black] (-0.3416876048222998,1.5) circle (1.5pt);
\draw[color=black] (-0.15,1.75) node {\large{$P_1$}};
\draw [fill=black] (-3.6583123951777,1.5) circle (1.5pt);
\draw[color=black] (-3.95,1.75) node {\large{$S_1$}};
\draw [fill=black] (0,-1) circle (1.5pt);
\draw[color=black] (0.2,-1.25) node {\large{$Q_1$}};
\draw [fill=black] (-4,-1) circle (1.5pt);
\draw[color=black] (-4.2,-1.25) node {\large{$R_1$}};
\draw [fill=black] (-2,1.5) circle (1.5pt);
\draw [fill=black] (-2,-1) circle (1.5pt);
\end{scriptsize}
\end{tikzpicture}
\caption{Illustration of the proof of Lemma~\ref{lem length}.}
\label{fig 5}
\end{figure}

\begin{proof}
Without loss of generality, we may assume that either $O_1$ and $O_2$ are on different sides of the line $AB$, or $O_1\in AB$ (otherwise, apply symmetry to $k_2$ with respect to $AB$ - this will not change the lengths of both $P_1Q_1$ and $P_2Q_2$), and we may also assume that $d(O_1,\ell_1) = h_1$ and $d(O_1,\ell_2) = h_2$ with $h_1\ge h_2$, see Figure~\ref{fig 5}. We have that 
\begin{align*}
    |P_1Q_1| &=\hspace{0.3em} \dfrac{h_1+h_2}{\sin(\angle P_1Q_1R_1)} =\hspace{0.3em} \dfrac{h_1+h_2}{\sin(\angle P_1O_1R_1/2)} \\ 
    &=\hspace{0.3em} \dfrac{h_1+h_2}{\sin(\pi/2 - \angle O_1P_1S_1/2+\angle O_1Q_1R_1/2)}\\ 
    &=\hspace{0.3em} \dfrac{h_1+h_2}{\cos(\angle O_1P_1S_1/2-\angle O_1Q_1R_1/2)},
\end{align*}
and also 
\begin{align*}
    |P_2Q_2| &=\hspace{0.3em} \dfrac{h_1+h_2}{\sin(\angle P_2Q_2R_2)} =\hspace{0.3em} \dfrac{h_1+h_2}{\sin(\angle P_2O_2R_2/2)}\\ 
    &=\hspace{0.3em} \dfrac{h_1+h_2}{\sin(\pi/2 - \angle O_2P_2S_2/2+\angle O_2Q_2R_2/2)}\\ 
    &=\hspace{0.3em} \dfrac{h_1+h_2}{\cos(\angle O_2P_2S_2/2-\angle O_2Q_2R_2/2)}.
\end{align*}
Note that since $h_1\ge h_2$ we have that
\begin{equation*}
    \angle O_1P_1S_1 = \arcsin(h_1/r_1) \ge \arcsin(h_2/r_1) = \angle O_1Q_1R_1
\end{equation*}
and 
\begin{equation*}
    \angle O_2P_2S_2 = \arcsin(h_1/r_2) \ge \arcsin(h_2/r_2) = \angle O_2Q_2R_2.
\end{equation*}
Moreover, standard analysis shows that the function $f:r\in [h_1,+\infty)\mapsto \arcsin(h_1/r) - \arcsin(h_2/r)$ is decreasing and therefore
\begin{equation*}
    \angle O_2P_2S_2-\angle O_2Q_2R_2 = f(r_2) \le f(r_1) = \angle O_1P_1S_1-\angle O_1Q_1R_1.
\end{equation*}
We conclude that
\begin{equation*}
    |P_2Q_2| = \dfrac{h_1+h_2}{\cos(f(r_2)/2)}\le \dfrac{h_1+h_2}{\cos(f(r_1)/2)} = |P_1Q_1|,
\end{equation*}
and the proof of the lemma is completed.
\end{proof}

\begin{corollary}\label{cor length}
In the setting of Lemma~\ref{lem length}, the length of the arc $\overset{\frown}{Q_1P_1}$ in $k_1$ is larger than or equal to the length of the arc $\overset{\frown}{P_2Q_2}$ in $k_2$.
\end{corollary}
\begin{proof}
This follows immediately from the fact that by Lemma~\ref{lem length} $|Q_1P_1|\ge |P_2Q_2|$, and that the curvature of the cycle $k_1$ is larger than or equal to the curvature of the cycle $k_2$ (recall that $r_1\le r_2$). 
\end{proof}

The above corollary can be used to bound the length of the boundary. The next lemma is the key observation that we will use for this purpose. 

\begin{lemma}\label{lem geom LB}
Fix a finite set of points $\{ O_i \}_{i\in \mathcal I}$ and a ball $\Bb' = \Bb(O', r')$, where $r'\le r/3$. Assume that the points $O'\cup \{ O_i \}_{i\in \mathcal I}$ are in general position. Then, the length of the boundary between the set
\begin{equation*}
    \Rr_1 := \Bb'\cap \left(\bigcup_{i\in\mathcal I} \Bb(O_i, r)\right)
\end{equation*}
and its complement $\Rr_2 := \Bb' \setminus \Rr_1$ is at most the length of the perimeter of $\Bb'$, that is, is at most $2\pi r'$.
\end{lemma}
\begin{proof}
For convenience, let $\partial \Rr$ denote the boundary of region $\Rr$ and let $\Bb_i = \Bb(O_i, r)$ for every $i\in \mathcal I$. Let us consider the following transformation of $\partial \Rr_1$. For any $i\in \mathcal I$ and any arc in $\partial \Rr_1 \cap \partial \Bb_i$, carry this arc along two rays starting at its endpoints, parallel to the ray $\vv{O'O_i}$, and in the same direction as this ray, to an arc of $\partial \Bb'$. For example, there are three such arcs in the top part of Figure~\ref{fig 4}, drawn in red; for example, the arc $\overset{\frown}{A_1X} \subseteq \partial \Bb_1$ is projected alongside $\vv{O'O_1}$ to the arc $\overset{\frown}{X_1A_1} \subseteq \partial \Bb'$ (note that we use the convention that arcs are always taken in anticlockwise direction).

First, let us note that by Corollary~\ref{cor length} the image of every arc in $\partial \Rr_1$ is at least as long as the arc itself. Hence, it remains to prove that the images of different arcs are disjoint. If $|\mathcal I| = 1$, then the statement trivially holds and so we may assume that $|\mathcal I| \ge 2$. Note also that it is sufficient to prove this fact for any pair of arcs in $\partial \Rr_1$ that belong to two different balls from the family $\{ \Bb_i\}_{i\in \mathcal I}$, say $\Bb_1$ and $\Bb_2$. Recall that $\Bb_1$ has center $O_1$ and suppose that it intersects $\Bb'$ in $A_1$ and $B_1$. Similarly, recall that $\Bb_2$ has center $O_2$ and suppose that it intersects $\Bb'$ in $A_2$ and $B_2$. See the bottom part of Figure~\ref{fig 4} for illustration.

Suppose that $\overset{\frown}{B_1A_1} \subseteq \partial \Bb'$ and $\overset{\frown}{B_2A_2} \subseteq \partial \Bb'$ have a non-empty intersection in $\Bb'$; otherwise, the statement clearly holds. Then, the arcs $\overset{\frown}{A_1B_1} \subseteq \partial \Bb_1$ and $\overset{\frown}{A_2B_2} \subseteq \partial \Bb_2$ intersect in a unique point $X\in \Bb_R$. Let  $X_1\in \partial \Bb'$ be the point such that $XX_1||O'O_1$ with $\vv{XX_1}$ having the same direction as $\vv{O'O_1}$. Similarly, let $X_2 \in \partial \Bb'$ be the point such that and $\vv{XX_2}||\vv{O'O_2}$ with $\vv{XX_2}$ having the same direction as $\vv{O'O_2}$ (see, again, the bottom part of Figure~\ref{fig 4}). Then, $\partial \Rr_1 \cap \partial \Bb_1$ is either contained in the arc $\overset{\frown}{A_1X}$ or in the arc $\overset{\frown}{XB_1}$ of $\partial \Bb_1$. We may assume that $\partial \Rr_1 \cap \partial \Bb_1$ is contained in the arc $\overset{\frown}{A_1X}$ of $\partial \Bb_1$, as the other case can be dealt with analogously. Then, $\partial \Rr_1 \cap \partial\Bb_2$ is contained in the arc $\overset{\frown}{XB_2}$ of $\partial \Bb_2$. In the rest of the proof, all arcs belong to $\partial \Bb'$. Our goal is to prove that the arcs $\overset{\frown}{X_1A_1}$ and $\overset{\frown}{B_2X_2}$ are disjoint. To show this we perform a continuous rotation of $\Bb_2$ around the point $X$ in the direction which decreases the length of the (directed) arc $\overset{\frown}{B_2B_1}$, until $\Bb_2$ coincides with $\Bb_1$. This operation decreases the length of the arc $\overset{\frown}{X_2X_1}$ as well. More importantly, at the end of this rotation when $\Bb_2$ coincides with $\Bb_1$, the arc $\overset{\frown}{B_1X_1}$ becomes the image of the arc $\overset{\frown}{B_2X_2}$. This proves that the arcs $\overset{\frown}{X_1A_1}$ and $\overset{\frown}{B_2X_2}$ of $\Bb'$ were initially disjoint. This finishes the proof of the lemma since $|\partial \Bb'|=2\pi r'$.
\end{proof}

\begin{figure}
\centering
\begin{minipage}{1.\textwidth}
\resizebox{30em}{!}{%
\hspace*{25em}
\definecolor{light gray}{rgb}{0.8,0.8,0.8}
\begin{tikzpicture}[line cap=round,line join=round,x=1cm,y=1cm]
\clip(-20,-7.2) rectangle (-4,2.5);
\draw [line width=0.8pt] (4-10,2) circle (6cm);
\draw [line width=0.8pt] (-1-10,-7) circle (6cm);
\draw [line width=0.8pt] (-4-10,-7) circle (6cm);
\draw [line width=0.8pt,dotted] (-2-10,0)-- (4-10,2);
\draw [line width=0.8pt,dotted] (-2-10,0)-- (-1-10,-7);
\draw [line width=0.8pt,dotted] (-2-10,0)-- (-4-10,-7);
\draw [line width=0.4pt,dash pattern=on 1pt off 1pt] (-1.1943336768604127-10,-1.0031479572997708)-- (-0.10519982175151557-10,-0.6401033389301384);
\draw [line width=0.4pt,dash pattern=on 1pt off 1pt] (-1.1943336768604127-10,-1.0031479572997708)-- (-1.0836778997437515-10,-1.7777383971163974);
\draw [line width=0.4pt,dash pattern=on 1pt off 1pt] (-2.5-10,-1.1905249806888747)-- (-2.389840983981042-10,-1.9616380928215822);
\draw [line width=0.4pt,dash pattern=on 1pt off 1pt] (-2.5-10,-1.1905249806888747)-- (-2.695603233892393-10,-1.8751362993122516);

\draw [shift={(-2-10,0)},line width=0.8pt, dash pattern=on 6pt off 6pt]  plot[domain=-0.32578402496009407:1.5707963267948966,variable=\t]({1*2*cos(\t r)+0*2*sin(\t r)},{0*2*cos(\t r)+1*2*sin(\t r)});

\draw [shift={(-2-10,0)},line width=0.8pt]  plot[domain=1.5707963267948966:3.668743702840922,variable=\t]({1*2*cos(\t r)+0*2*sin(\t r)},{0*2*cos(\t r)+1*2*sin(\t r)});

\draw [shift={(-2-10,0)},line width=0.8pt, dash pattern=on 6pt off 6pt]  plot[domain=4.516212593746767:5.188314208713056,variable=\t]({1*2*cos(\t r)+0*2*sin(\t r)},{0*2*cos(\t r)+1*2*sin(\t r)});

\draw [shift={(-2-10,0)},line width=0.8pt]  plot[domain=4.357163671019901:4.516212593746767,variable=\t]({1*2*cos(\t r)+0*2*sin(\t r)},{0*2*cos(\t r)+1*2*sin(\t r)});

\draw [shift={(-2-10,0)},line width=0.8pt, dash pattern=on 6pt off 6pt]  plot[domain=3.668743702840922:4.357163671019901,variable=\t]({1*2*cos(\t r)+0*2*sin(\t r)},{0*2*cos(\t r)+1*2*sin(\t r)});

\draw [shift={(-2-10,0)},line width=0.8pt]  plot[domain=-0.32578402496009407:5.188314208713056-6.283185307179586,variable=\t]({1*2*cos(\t r)+0*2*sin(\t r)},{0*2*cos(\t r)+1*2*sin(\t r)});

\draw [shift={(-4-10,-7)},line width=0.8pt,color=light gray]  plot[domain=1.318116071652818:1.5255288039532051,variable=\t]({1*6*cos(\t r)+0*6*sin(\t r)},{0*6*cos(\t r)+1*6*sin(\t r)});

\draw [shift={(-1-10,-7)},line width=0.8pt,color=light gray]  plot[domain=1.6031909385172858:1.8234765819369754,variable=\t]({1*6*cos(\t r)+0*6*sin(\t r)},{0*6*cos(\t r)+1*6*sin(\t r)});
\draw [shift={(4-10,2)},line width=0.8pt,color=light gray]  plot[domain=3.141592653589793:3.665797359877627,variable=\t]({1*6*cos(\t r)+0*6*sin(\t r)},{0*6*cos(\t r)+1*6*sin(\t r)});

\begin{scriptsize}
\draw [fill=black] (-2-10,0) circle (0.8pt);
\draw [fill=black] (4-10,2) circle (0.8pt);
\draw [fill=black] (-2-10,2) circle (0.8pt);
\draw [fill=black] (-1-10,-7) circle (0.8pt);
\draw [fill=black] (-4-10,-7) circle (0.8pt);
\draw [fill=black] (-1.1943336768604127-10,-1.0031479572997708) circle (0.8pt);
\draw [fill=black] (-2.5-10,-1.1905249806888747) circle (0.8pt);
\draw [fill=black] (-3.7284876133285287-10,-1.006146396191849) circle (0.8pt);
\draw [fill=black] (-0.10519982175151557-10,-0.6401033389301384) circle (0.8pt);
\draw [fill=black] (-1.0836778997437515-10,-1.7777383971163974) circle (0.8pt);
\draw [fill=black] (-2.389840983981042-10,-1.9616380928215822) circle (0.8pt);
\draw [fill=black] (-2.695603233892393-10,-1.8751362993122516) circle (0.8pt);

\draw[color=black] (-1.695603233892393-10,2.25) node {\Large{$A_1$}};
\draw[color=black] (-11,-0.7) node {\Large{$X$}};
\draw[color=black] (-9.7,-0.45) node {\Large{$X_1$}};

\draw[color=black] (-12.3,0) node {\Large{$O'$}};
\draw[color=black] (-5.6,2) node {\Large{$O_1$}};

\end{scriptsize}
\end{tikzpicture}
}
\end{minipage}

\vspace{3em}

\begin{minipage}{1.\textwidth}
\resizebox{30em}{!}{%
\hspace*{25em}
\begin{tikzpicture}[line cap=round,line join=round,x=1cm,y=1cm]
\clip(-10,-9) rectangle (6,2.5);
\draw [line width=0.8pt] (-2,0) circle (2cm);
\draw [line width=0.8pt] (4,2) circle (6cm);
\draw [line width=0.8pt] (-1,-7) circle (6cm);
\draw [line width=0.8pt,dotted] (-2,0)-- (4,2);
\draw [line width=0.8pt,dotted] (-2,0)-- (-1,-7);
\draw [line width=0.4pt,dash pattern=on 1pt off 1pt] (-1.1943336768604127,-1.0031479572997708)-- (-0.10519982175151557,-0.6401033389301384);
\draw [line width=0.4pt,dash pattern=on 1pt off 1pt] (-1.1943336768604127,-1.0031479572997708)-- (-1.0836778997437515,-1.7777383971163974);
\begin{scriptsize}
\draw [fill=black] (-2,0) circle (0.8pt);
\draw[color=black] (-2.2722781934821734,0.24331195480697174) node {\Large{$O'$}};
\draw [fill=black] (4,2) circle (0.8pt);
\draw[color=black] (4.043963628873572,1.6555348286194613) node {\Large{$O_1$}};
\draw [fill=black] (-2,2) circle (0.8pt);
\draw[color=black] (-1.6143506613613818,2.2013254331753506) node {\Large{$A_1$}};
\draw [fill=black] (-1,-7) circle (0.8pt);
\draw[color=black] (-0.6221707606726839,-6.725601027004884) node {\Large{$O_2$}};
\draw [fill=black] (-1.1943336768604127,-1.0031479572997708) circle (0.8pt);
\draw[color=black] (-1.1327479656877052,-0.6557055276581663) node {\Large{$X$}};
\draw [fill=black] (-0.10519982175151557,-0.6401033389301384) circle (0.8pt);
\draw[color=black] (0.3068467217924565,-0.40410283198449054) node {\Large{$X_1$}};
\draw [fill=black] (-1.0836778997437515,-1.7777383971163974) circle (0.8pt);
\draw[color=black] (-1.0879828517904713,-2.188419660238087) node {\Large{$X_2$}};
\draw [fill=black] (-3.3472196960489997,-1.4781742422927142) circle (0.8pt);
\draw[color=black] (-3.34603930328265,-1.9041887782320792) node {\Large{$B_2$}};
\draw [fill=black] (-0.292780303951,-1.0418257577072858) circle (0.8pt);
\draw[color=black] (0.02475597388122008,-1.1462612461112889) node {\Large{$A_2$}};
\draw [fill=black] (-0.8,-1.6) circle (0.8pt);
\draw[color=black] (-0.5037304832225648,-1.778639417849646) node {\Large{$B_1$}};
\end{scriptsize}
\end{tikzpicture}
}
\end{minipage}
\caption{The top figure shows how the internal grey arcs are carried over to the external dashed arcs. The bottom figure shows an illustration of the proof that the external arcs are disjoint.}
\label{fig 4}
\end{figure}
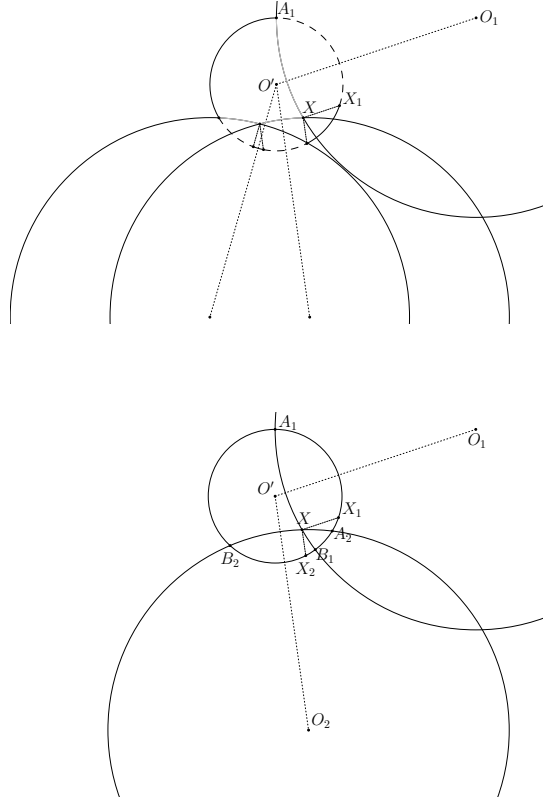

We will also need the following fact that has been known for centuries and by now has become part of the mathematical folklore.

\begin{lemma}[Folklore; see for example~\cite{Oss}]\label{lem iso}
Out of all connected open sets in the plane with a given perimeter, the circle has the largest area. In other words, each connected open set of perimeter $\ell$ has area at most $\ell^2 / 4 \pi$.
\end{lemma}

Recall that $\{ O_i \}_{i\in \mathcal I}$ is assumed to be a finite set of points that partitions $\Bb_R$ into 
$$    
\Rr_1 := \Bb_R \cap \left(\bigcup_{i\in \mathcal I} \Bb(O_i, r)\right) \qquad \text{ and } \qquad \Rr_2 := \Bb_R \setminus \Rr_1.
$$
For example, in the bottom part of Figure~\ref{fig 4}, the region $\Rr_1$ is bounded by the arcs $\overset{\frown}{A_1 X}$, $\overset{\frown}{XB_2}$ and $\overset{\frown}{B_2A_1}$ from the circles with centers $O_1$, $O_2$ and $O'$, respectively.
(Eventually, $\{ O_i \}_{i\in \mathcal I}$ will be fixed to be $D_{k-1}(v)$, the set of vertices that are at distance $k-1$ from $v$ for some $k \in \N$.) 
Note that $\Rr_2$ does not need to be a connected set. However, since the number of balls $(\Bb(O_i, r))_{i\in \mathcal I}$ is finite, the number of contact points between their boundaries is also finite, and therefore the number of connected components of $\Rr_2$ must also be finite. A component of $\Rr_2$ will be called \emph{large} if its boundary has length more than $6 \pi \varepsilon$, and it will be called \emph{small} otherwise.

Let us first concentrate on small components and consider the union of them. Suppose that for some $k\in \mathbb N$ there are $k$ small components with lengths of their boundaries $\ell_1, \ell_2, \ldots, \ell_k$. By Lemma~\ref{lem geom LB}, we get that 
\begin{equation*}
    \sum_{i=1}^k \ell_i \le \dfrac{2\pi r}{3},
\end{equation*}
and thus using Lemma~\ref{lem iso} we deduce that the area of the union of all small regions is at most
\begin{equation}\label{smallarea}
    \sum_{i=1}^k \frac{\ell^2_i}{4\pi} \le \dfrac{3\varepsilon}{2} \sum_{i=1}^k \ell_i \le \pi r\varepsilon. 
\end{equation}

Now we may concentrate on large components. Let $\gamma\subseteq \mathbb R^2$ be a closed curve. The \emph{$\varepsilon$-tube $t_{\varepsilon}(\gamma)$} around $\gamma$ is the set of points $Q \in \mathbb R^2$ such that $d_E(\gamma, Q)\le \varepsilon$. Moreover, for any arc $\overset{\frown}{a}$ of a circle $c$ with radius at least $r/3 > 2\varepsilon$, define the \emph{$\varepsilon$-cut tube $t^c_{\varepsilon}(\overset{\frown}{a})$} around $\overset{\frown}{a}$ as the intersection of $t_{\varepsilon}(c)$ with the sector of $c$, corresponding to the arc $\overset{\frown}{a}$. In the next observation, the diameter of an arc is the longest (Euclidean) distance between some two points in this arc.

\begin{observation}\label{ob arc}
Let $A,B$ be two points inside a ball $\Bb$ with radius $r_1$. Let $\overset{\frown}{AB}$ be an arc between $A$ and $B$ with diameter $|AB|$, which is part of a circle $c$ with radius $r_2 > r_1$. Then, $\overset{\frown}{AB}\subseteq \Bb$.
\end{observation}
\begin{proof}
Since the radius of $c$ is larger than the radius of $\Bb$ and $A,B\in \Bb$, $c$ and $\partial \Bb$ must intersect in two points. Thus, either the arc $\overset{\frown}{AB}$ is entirely contained in $\Bb$ or it starts in $\Bb$, contains every point in $c \cap \Bb^c$ and then ends in $\Bb$. In the second case, the diameter of the arc $\overset{\frown}{AB}$ would be $2r_2 > 2r_1$, which would lead to a contradiction since $|AB|$ is clearly at most $2r_1$. Thus, $\overset{\frown}{AB}\subseteq \Bb$.
\end{proof}

\begin{lemma}\label{lem cut tubes}
For every large component $S$ of $\Rr_2$ we have that the area of $t_{\varepsilon}(\partial S)$ is at most the sum of the areas of the $\varepsilon$-cut tubes around the arcs participating in $\partial S$, that is, at most $2\varepsilon |\partial S|$.
\end{lemma}

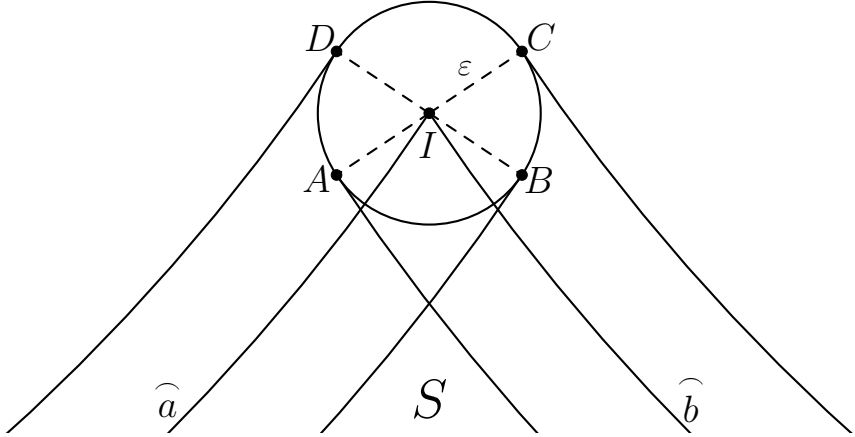
\begin{figure}
\centering
\begin{tikzpicture}[scale=4,line cap=round,line join=round,x=1cm,y=1cm]
\clip(-0.7596414975661192,-4.061317504723573) rectangle (2.8979409408262207,-2.335395791607204);
\draw [line width=0.8pt] (1,-3) circle (0.3704212547378905cm);
\draw [shift={(-5,1)},line width=0.8pt]  plot[domain=4.905600464232496:5.695182703632018,variable=\t]({1*6.8406812961900885*cos(\t r)+0*6.8406812961900885*sin(\t r)},{0*6.8406812961900885*cos(\t r)+1*6.8406812961900885*sin(\t r)});
\draw [shift={(-5,1)},line width=0.8pt]  plot[domain=4.905600464232496:5.695182703632018,variable=\t]({1*7.581523805665869*cos(\t r)+0*7.581523805665869*sin(\t r)},{0*7.581523805665869*cos(\t r)+1*7.581523805665869*sin(\t r)});
\draw [shift={(7,1)},line width=0.8pt]  plot[domain=3.7295952571373605:4.519177496536884,variable=\t]({1*7.581523805665869*cos(\t r)+0*7.581523805665869*sin(\t r)},{0*7.581523805665869*cos(\t r)+1*7.581523805665869*sin(\t r)});
\draw [shift={(7,1)},line width=0.8pt]  plot[domain=3.7295952571373605:4.519177496536884,variable=\t]({1*6.840681296190089*cos(\t r)+0*6.840681296190089*sin(\t r)},{0*6.840681296190089*cos(\t r)+1*6.840681296190089*sin(\t r)});
\draw [shift={(-5,1)},line width=0.8pt]  plot[domain=4.905600464232496:5.695182703632018,variable=\t]({1*7.211102550927981*cos(\t r)+0*7.211102550927981*sin(\t r)},{0*7.211102550927981*cos(\t r)+1*7.211102550927981*sin(\t r)});
\draw [shift={(7,1)},line width=0.8pt]  plot[domain=3.7295952571373605:4.519177496536884,variable=\t]({1*7.211102550927979*cos(\t r)+0*7.211102550927979*sin(\t r)},{0*7.211102550927979*cos(\t r)+1*7.211102550927979*sin(\t r)});
\draw [line width=0.8pt,dash pattern=on 4pt off 4pt] (1,-3)-- (1.3082091140336543,-2.794527257310897);
\draw [line width=0.8pt,dash pattern=on 4pt off 4pt] (1,-3)-- (1.3082091140336551,-3.2054727426891034);
\draw [line width=0.8pt,dash pattern=on 4pt off 4pt] (1,-3)-- (0.6917908859663452,-3.205472742689103);
\draw [line width=0.8pt,dash pattern=on 4pt off 4pt] (1,-3)-- (0.6917908859663457,-2.7945272573108975);

\begin{scriptsize}
\draw [fill=black] (1,-3) circle (0.5pt);
\draw[color=black] (0.9948550783502065,-3.106908673345728) node {\Large{$I$}};
\draw [fill=black] (1.3082091140336551,-3.2054727426891034) circle (0.5pt);
\draw[color=black] (1.3648983496294522,-3.2226452815453372) node {\Large{$B$}};
\draw [fill=black] (1.3082091140336543,-2.794527257310897) circle (0.5pt);
\draw[color=black] (1.3720384593895,-2.745445072786341) node {\Large{$C$}};
\draw [fill=black] (0.6917908859663457,-2.7945272573108975) circle (0.5pt);
\draw[color=black] (0.6376716973109126,-2.7483025590663352) node {\Large{$D$}};
\draw [fill=black] (0.6917908859663452,-3.205472742689103) circle (0.5pt);
\draw[color=black] (0.6248142110309185,-3.2197877952653435) node {\Large{$A$}};

\draw[color=black] (0.13,-3.95) node {\Large{$\overset{\frown}{a}$}};

\draw[color=black] (1.87,-3.95) node {\Large{$\overset{\frown}{b}$}};

\draw[color=black] (1.119156933509925,-2.849742120026143) node {\large{$\varepsilon$}};

\draw[color=black] (1,-3.95) node {\huge{$S$}};

\end{scriptsize}
\end{tikzpicture}
\caption{Illustration of the proof of Lemma~\ref{lem cut tubes}. Here, $I$ is an intersection point of the arcs $\overset{\frown}{a}$ and $\overset{\frown}{b}$ on the boundary of the large component $S$. The points $O_1$ and $O_2$ are not shown since they are too far and $S$ is contained between $\overset{\frown}{a}$ and $\overset{\frown}{b}$.}
\label{fig 6}
\end{figure}

\begin{proof}
The claim is trivial if $\Rr_2 = \Bb_R$. Otherwise, consider an arbitrary intersection point $I$ of two circles $c_1$ and $c_2$ with centers $O_1$ and $O_2$, respectively. Let $\overset{\frown}{a}$ and $\overset{\frown}{b}$ be the two arcs in $\partial S$, contained in the circles $c_1$ and $c_2$, respectively, which contain $I$ as an endpoint. Let $A,I,C,O_1$ be collinear points, lying on the line $IO_1$ in this order and such that $|AI| = |IC| = \varepsilon$. Let also $B,I,D,O_2$ be collinear points, lying on the line $IO_2$ in this order and such that $|BI| = |ID| = \varepsilon$. Then, define the \emph{internal sector at $I$}, denoted by $IS_S(I, \varepsilon)$, to be the sector $AIB$ of the ball $\Bb(I, \varepsilon)$, and also define the \emph{external sector at $I$}, denoted by $ES_S(I, \varepsilon)$, to be the sector $CID$ of the ball $\Bb(I, \varepsilon)$. Then, the tube $t_{\varepsilon}(\partial S)$ is obtained as a union of all cut tubes of the arcs in $\partial S$ and the external sectors at all intersection points of neighbouring arcs - see Figure~\ref{fig 6}. 

We need the following two claims before proceeding with the proof of the lemma.

\begin{claim}\label{claim 1a}
There exist two internal sectors without common points.
\end{claim}
\begin{proof}[Proof of the claim]
Suppose for a contradiction that each pair of internal sectors intersect. Fix one internal sector with center $I$. Then, by the triangle inequality every intersection point of two neighbouring arcs in $\partial S$ is at distance at most $2\varepsilon$ from $I$. Since $r/3 > 2\varepsilon$, by Observation~\ref{ob arc} every arc in $\partial S$ is contained in $\Bb(I, 2\varepsilon)$ and therefore $S\subseteq \Bb(I, 2\varepsilon)$. If $\partial \Bb_R\cap \partial S = \emptyset$, Lemma~\ref{lem geom LB} with $\Bb' = \Bb(I, 2\varepsilon)$ implies that the perimeter of $S$ is at most $|\partial \Bb(I, 2\varepsilon)| = 4\pi \varepsilon$, contradicting with the fact that $S$ is a large component. 
Otherwise, the perimeter of $S$ is bounded from above by the sum of $|\partial S\cap \partial \Bb_R|$ and the perimeter of the region containing $S$, in $\Bb(I, 2\varepsilon)\cap \left(\bigcup_{i\in \mathcal I} \Bb_i\right)^c$. Roughly speaking, the region described above is obtained by ``taking out'' the ball $\Bb_R$. First, note that $\partial S\cap \partial \Bb_R\subseteq \Bb(I, 2\varepsilon)\cap \partial \Bb_R$ and therefore, since the curvature of $\partial \Bb_R$ is smaller than the curvature of $\partial \Bb(I, 2\varepsilon)$, we have $|\partial S\cap \partial \Bb_R|\le |\Bb(I, 2\varepsilon)\cap \partial \Bb_R|\le |\partial \Bb(I, 2\varepsilon)|/2 = 2\pi \varepsilon$. Second, one may apply Lemma~\ref{lem geom LB} to $\Bb(I, 2\varepsilon)\cap \left(\bigcup_{i\in \mathcal I} \Bb_i\right)^c$ with $\Bb' = \Bb(I, 2\varepsilon)$ to conclude that the perimeter of $\Bb(I, 2\varepsilon)\cap \left(\bigcup_{i\in \mathcal I} \Bb_i\right)^c$ is at most $|\partial \Bb(I, 2\varepsilon)| = 4\pi \varepsilon$. In total we get $|\partial S|\le 6\pi \varepsilon$, which again is a contradiction with the fact that $S$ is a large component.
\end{proof}

\begin{claim}\label{claim 2}
Let $\ell\in \mathbb N$. If a point $P$ is contained in $\ell$ internal sectors, it must be contained in the $\varepsilon$-cut tubes around at least $\ell+1$ of the arcs in $\partial S$.
\end{claim}
\begin{proof}[Proof of the claim]
By Claim~\ref{claim 1a}, there exists an internal sector $IS_S(I, \varepsilon)$, which does not contain $P$. Let us enumerate the arcs along $\partial S$, starting from one of the arcs, incident to $I$, and finish with the other arc, incident to $I$. Let $I_1, I_2,\dots, I_{\ell}$ be the points, for which $P\in \cap_{i\in [\ell]} IS_S(I_i, \varepsilon)$. Then, $P$ is contained in the $\varepsilon$-cut tube around every arc of $\partial S$, incident to any of $I_1, I_2, \dots,I_{\ell}$, and there are at least $\ell+1$ such arcs.
\end{proof}

By Claim~\ref{claim 2}, one may directly deduce that
\begin{align*}
    \sum_{\overset{\frown}{a}\in \partial S} |t^c_{\varepsilon}(\overset{\frown}{a})| & \ge\hspace{0.3em}
    \left| \bigcup_{\overset{\frown}{a}\in \partial S} t^c_{\varepsilon}(\overset{\frown}{a})\right| + \sum_{I\, :\, I = \overset{\frown}{a} \cap \overset{\frown}{b}; \overset{\frown}{a}, \overset{\frown}{b} \in \partial S} |IS_S(I, \varepsilon)|\\
    & =\hspace{0.3em}
    \left| \bigcup_{\overset{\frown}{a}\in \partial S} t^c_{\varepsilon}(\overset{\frown}{a})\right| + \sum_{I\, :\, I = \overset{\frown}{a} \cap \overset{\frown}{b}; \overset{\frown}{a}, \overset{\frown}{b} \in \partial S} |ES_S(I, \varepsilon)|
    \ge |t_{\varepsilon}(S)|.
\end{align*}
The proof of the lemma is finished.
\end{proof}

\begin{corollary}\label{cor cut tubes}
The union of all $\varepsilon$-tubes around the boundaries of large components has area at most $4\pi \varepsilon r/3$.
\end{corollary}
\begin{proof}
This follows from Lemma~\ref{lem cut tubes}, applied for every large component of $\Rr_2$, and Lemma~\ref{lem geom LB}, which states that the union of the boundaries of these components has length at most $2\pi r/3$.
\end{proof}

Let us now come back to the proof of the lower bound for the localization number of dense graphs. 

\begin{proof}[Proof of Theorem~\ref{main: lower bound}]
Fix $r=r(n) \ge \log n$ and $\eps=\eps(n) = (\log n / r)^{1/3} \le 1$. Let $\Bb_R$ be the ball of radius $r/3$ centered in the center $O$ of the square $[0, \sqrt{n}]^2$. It follows from Lemma~\ref{lem:special_family} that a.a.s.\ there exists an $\eps$-special family of pairs of vertices of size $r^2 \eps^2 / 100$ in $\Bb_R$. We will show that a.a.s.\ a single sensor placed on a vertex $v$ cannot distinguish more than $100 \eps^3 r$ special pairs which will imply the desired lower bound of $(r^2 \eps^2 / 100) / (100 \eps^3 r) = 10^{-4} \, r / \eps = 10^{-4} \, r^{4/3} / \log^{1/3} n$.

Indeed, suppose that the cops must use less than $10^{-4} \, r^{4/3} / (\log n)^{1/3}$ sensors in each round. Using the notation from Section~\ref{sec:perfect_information}, when the cops start the game by putting their sensors on $S_1$, at least two vertices (namely, some special pair) in $\Bb_R$ have the same $S_1$-signature. The robber may choose the equivalence class $R^1_{j_1}$ these two vertices belong to and remain undetected in the very first round. Suppose now that $R^{i-1}_{j_{i-1}}$ contains at least two vertices from $\Bb_R$. In round $i$, once the cops choose $S_i$, we get the partition $N[R^{i-1}_{j_{i-1}}]=R^i_{1}\cup R^i_{2}\cup \ldots \cup R^i_{\ell_i}$ with every vertex in $R^i_{j}$ having the same $S_i$-signature. Since the ball $\Bb_R$ has radius $r/3$, $N[R^{i-1}_{j_{i-1}}]$ includes \emph{all} vertices in $\Bb_R$. Hence, again, the robber may choose some $R^i_{j_i}$ of size at least 2 as there is at least one special pair of vertices in $\Bb_R$ with the same $S_i$-signature. It follows that $|R^i_{j_i}| \ge 2$ for all $i$ and so the robber has a winning strategy. 

\medskip

It remains to show that a.a.s.\ a single sensor placed on a vertex $v$ cannot distinguish more than $100 \eps^3 r$ special pairs. Clearly, if $v$ is in $\Bb_R$, then it can distinguish at most one special pair, the one including the vertex $v$ itself. Hence, we may concentrate on sensors placed on vertices outside of the ball the robber is hiding at. To that end, we will use de-Poissonization technique as explained in Section~\ref{sec:Poisson} and show that the desired property holds with probability $1-o(n^{-2})$ for a given vertex $v$ outside of $\Bb_R$. The desired conclusion will hold by a union bound over all vertices. 

Let $v$ be any vertex of $G\in \RT$ that is outside of the ball $\Bb_R$. We will carefully expose the graph in a \emph{breadth-first-search} fashion. Recall that $D_i(v)$ denotes the set of vertices at graph distance $i$ from $v$. We start with $D_0(v) = \{v\}$. Iteratively, as long as no vertex in $D_i(v)$ is at Euclidean distance at most $r$ from $\Bb_R$ we do the following. Since vertices in $D_{i+1}(v)$ must belong to 
$$
U(D_{i}(v)) := \bigcup_{u \in D_{i}(v)} \Bb(u,r),
$$ 
we expose all vertices in the part of $U(D_{i}(v))$ that is not exposed yet. Vertices that are found there form the set $D_{i+1}(v)$. We stop the process prematurely if no vertex is found in $U(D_{i}(v))$; in this case $v$ does not distinguish \emph{any} special pair and so the desired property holds. Suppose that for some $k \in \N$ we stopped the process because for the first time some vertex $w \in D_{k-1}(v)$ is at Euclidean distance at most $r$ from $\Bb_R$. If $w$ is in fact at distance at most $r/3$ from $\Bb_R$, then all points in $\Bb_R$ are at distance at most $r$ from $w$. It follows that, despite the fact that we did not expose the ball yet, we can safely claim that all vertices in $\Bb_R$ will end up in $D_k(v)$. On the other hand, if no vertex in $D_{k-1}(v)$ is at Euclidean distance at most $r/3$ from $\Bb_R$, then we may pick the point $A$ that is on the segment between $w$ and the center $O$ of the ball $\Bb_R$ and at distance, say, $r/2$ from $O$.
By Observation~\ref{obs:vertices}, we may assume that there is a vertex at distance at most $2 \sqrt{\log n} = o(r)$ from $A$. (This is a standard technique in the theory of random graphs but it is quite delicate. We wish to use the properties guaranteed a.a.s.\ by Observation~\ref{obs:vertices}, but we also wish to avoid working in a conditional probability space, as doing so would make the necessary probabilistic computations intractable. Thus, we will work in the unconditional probability space but in our argument we assume that the properties mentioned in the observation hold. Since these properties hold a.a.s., the probability of the set of outcomes in which our argument does \emph{not} apply to is $o(1)$, and thus can be safely excised at the end of the argument.) This vertex is not only adjacent to $w$ but also all points in $\Bb_R$ are at distance at most $r$ from it. Hence, this time we can safely claim that all vertices in $\Bb_R$ will end up in $D_k(v) \cup D_{k+1}(v)$. 

Let us summarize the current situation: $D_{k-1}(v)$ is a set of vertices at distance $k-1$ from $v$ that partitions $\Bb_R$ into 
$$    
\Rr_1 := \Bb_R \cap \left(\bigcup_{w \in D_{k-1}(v)} \Bb(w, r)\right) \qquad \text{ and } \qquad \Rr_2 := \Bb_R \setminus \Rr_1.
$$
The region $\Rr_1$ is non-empty but $\Rr_2$ might be empty. The ball $\Bb_R$ is not exposed yet but we do know that all vertices in $\Rr_1$ (if there are any) will end up in $D_k(v)$ and all vertices in $\Rr_2$ (again, if there are any) will end up in $D_{k+1}(v)$. In order for a pair of vertices $(a,b)$ to be distinguished by $v$, one of the two vertices (say, $a$) has to be in $\Rr_1$ and the other one (say, $b$) has to be in $\Rr_2$. More importantly, if $a$ and $b$ are at distance at most $\eps$ from each other, $b$ has to belong to some small component of $\Rr_2$ or to some $\eps$-tube around the boundary of some large component (but still in $\Rr_2$).

Let us expose vertices in $\Rr_2$. By~\eqref{smallarea} and Corollary~\ref{cor cut tubes}, we get that the total number of vertices in all small components and in the union of all $\eps$-tubes around the boundaries of large components is stochastically bounded from above by the random variable $X \sim \textrm{Po}(\lambda)$ with $\lambda := \pi \eps r + 4 \pi \eps r / 3 = 7 \pi \eps r / 3$. We have
\begin{eqnarray}
\Prob \Big( X \ge 2 \lambda \Big) &=& \sum_{i \ge 2\lambda} \frac {\lambda^i}{i!} e^{-\lambda} ~~\le~~ 2 \ \Prob \Big( X = 2 \lambda \Big) ~~=~~ 2 \ \frac {\lambda^{2 \lambda}}{(2 \lambda)!} e^{-\lambda} \nonumber \\
&\le& 2 \ \frac {\lambda^{2 \lambda}}{(2 \lambda/e)^{2\lambda}} e^{-\lambda} = 2 \ \left( \frac {e}{4} \right)^{\lambda} ~~\le~~ 2 \ \exp \left( - \frac {\lambda}{3} \right). \label{eq:Poisson}
\end{eqnarray}
Since $r \ge \log n$, we get that $\lambda = \frac {7 \pi}{3} \eps r  = \frac {7 \pi}{3} r^{4/3} / \log^{1/3} n \ge \frac {7 \pi}{3} \log n > 7 \log n$. It follows from~(\ref{eq:Poisson}) that with probability $1-o(n^{-2})$, $X \le 2 \lambda$. Each vertex $b$ that appears in this region eliminates at most one special pair but this itself is not enough to get the desired bound. 

We condition on the event that there are at most $2 \lambda = 14 \pi \eps r / 3$ vertices $b$ in $\Rr_2$ that can potentially participate in $\eps$-special pairs and argue as follows. Since the associated vertices $a$ have to be not only in $\Rr_1$ but also at distance at most $\eps$ from some vertex $b$ in $\Rr_2$, we expose vertices in $\Rr_1$ and check how many of them are close to some vertex~$b$ in $\Rr_2$. The number of such vertices $a$ is stochastically bounded from above by the random variable $Y \sim \textrm{Po}(\xi)$ with $\xi := (14 \pi \eps r / 3) (\pi \eps^2) = 14 \pi^2 \eps^3 r / 3 = (14 \pi^2 / 3) \log n > 46 \log n$. It follows from~(\ref{eq:Poisson}) that with probability $1-o(n^{-2})$, $Y \le 2 \xi = 28 \pi^2 \eps^3 r / 3 < 100 \eps^3 r$. Each such vertex $a$ that appears eliminates at most one special pair, and so the desired bound holds, and the proof is finished.
\end{proof}

Note that the previous proof gives the lower bounds of Part 1 and Part 2 in Theorem~\ref{thm:main}. Let us now consider sparser graphs. We first adjust the argument used above that gives a matching lower bound for $r$ very close to $\log n$ (Part 3) in Theorem~\ref{thm:main}), namely, we assume first that
$$
\frac {\log n}{(\log \log n)^{1/2} \log \log \log n} \le r \le \log n.
$$
By fixing $\eps=1$, we argue as in the proof of Lemma~\ref{lem:special_family} that a.a.s.\ there exists a $1$-special family of pairs of vertices of size $r^2 / 100$. The argument then proceeds as before: the total number of vertices in $\Rr_2$ in all small components and in the union of all $1$-tubes around the boundaries of large components is stochastically bounded from above by the random variable $X \sim \textrm{Po}(\lambda)$ with $\lambda := 7 \pi r / 3$. This time we fix 
$$
\beta := \frac {200 \log n}{\log (e \log n / r)} > 2 \lambda
$$
and notice that 
\begin{align*}
\Prob \Big( X \ge \beta \Big) &\le 2 \, \Prob \Big( X = \beta \Big) = 2 \, \frac { \lambda^\beta }{\beta!} e^{-\lambda} \le \frac {\lambda^\beta} {(\beta/e)^\beta} = \left( \frac {7 e \pi r}{3 \beta} \right)^\beta = \exp \left( - \beta \log \left( \frac {3 \beta}{7 e \pi r} \right) \right)\\
&=\hspace{0.3em} \exp\left(-\dfrac{200\log n}{\log(e\log n/r)}\log\left(\dfrac{600\log n/\log(e\log n/r)}{7e\pi r}\right)\right)\\
&\le \exp\left(-\dfrac{200\log n}{\log(e\log n/r)}\log\left(\dfrac{e\log n/r}{\log(e\log n/r)}\right)\right)\\
&=\hspace{0.3em} \exp\left(-200\log n\left(1 - \dfrac{\log\log(e\log n/r)}{\log(e\log n/r)}\right)\right) = o(1/n^2).
\end{align*}
Arguing as before, we get the lower bound of $(r^2 / 100) / \beta = \Theta( r^2 \log (e \log n / r) / \log n )$. Note that since $\eps = 1$ (so $r\eps = r\eps^3$), we have that the order of the number of points in $\Rr_2$ that $X$ counts would be comparable to the number of $1$-special pairs that one sensor may distinguish, and therefore the last stage in the proof of Part 1 and Part 2 of Theorem~\ref{thm:main} will not contribute and may be omitted.

\bigskip

Let us now concentrate on even sparser graphs for which we use a different argument. Tessellate the torus $\mathcal T_n$ into a family of squares $(S_i)_{i\in \mathcal I}$ of side lengths $3r$ and with centers $(O_i)_{i\in \mathcal I}$. Then, for every $i\in \mathcal I$, let $\Bb_i$ be the ball of radius $\log n/16r$ with center $O_i$. Finally, for every $i\in \mathcal I$, let $\Rr_i$ be the set of points $P$, for which $\Cc(P, r)\cap \Bb_i\neq \emptyset$.
The idea of this part is the following: for every square $S_i$ of the tessellation and the ball $\Bb_i$ inside $S_i$, note that the region $\Rr_i$ consists of the points that may distinguish some vertices inside $\Bb_i$. We will show that for many squares $S_i$, the corresponding region $\Rr_i$ contains no vertex of $G$, and at the same time there is one such square $S_i$ with many vertices inside $\Bb_i$. Since the robber can jump from one vertex to another inside the ball, she can only be trapped by putting a sensor on each vertex inside the ball. We fill in the details and start with the following preliminary result:
\begin{lemma}\label{lem:areari}
For every $i\in \mathcal I$, $\Rr_i$ has area $\pi \log n/4$ and is disjoint from $\Bb_i$. 
\end{lemma}
\begin{proof}
For every $i\in \mathcal I$, $\Rr_i$ consists of all points at distance between $r - \log n/16r$ and $r + \log n/16r$ from $O_i$, so $|\Rr_i| = \pi(r + \log n/16r)^2 - \pi(r - \log n/16r)^2 = \pi\log n/4$. Moreover, $2\cdot \log n/16r < r$ since $r\ge \sqrt{\log n}$, so $\Rr_i\cap \Bb_i = \emptyset$.
\end{proof}

\begin{proof}[Proof of the lower bound of Theorem~\ref{thm:main}, Part 4)]

Expose the region $\bigcup_{i\in \mathcal I} \Rr_i\subseteq \bigcup_{i\in \mathcal I} S_i\setminus \Bb_i$. Note that the regions $(\Rr_i)_{i\in \mathcal I}$ are disjoint and for any $i\in \mathcal I$ the probability that no vertex of $G\in \mathcal T(n,r)$ falls into $\Rr_i$ is by Lemma~\ref{lem:areari}, $\exp(-\pi \log n/4) = 1/n^{\pi/4}$. Let $\mathcal J$ be the family of indices $i$, for which $\Rr_i$ does not contain any vertices of $G$. We conclude that the family of variables $(\mathds 1_{\Rr_i\cap V(G)=\emptyset})_{i\in \mathcal I}$ consists of independent Bernoulli variables with parameter $1/n^{\pi/4}$, so by Chernoff's bound $|\mathcal J| \ge n^{1-\pi/4}/18 r^2$ with probability $1-o(1/\sqrt{n})$. We condition on this event.

Now, what remains is to give a lower bound that holds with probability $1-o(1/\sqrt{n})$ for the maximum of the $|\mathcal J|\ge n^{1-\pi/4}/18 r^2\ge n^{0.2}$ (the last inequality holds for every large enough $n$) Poisson variables with mean $\lambda = \pi \log^2 n/(16 r)^2$, representing the number of vertices of $G$ in the balls $(\Bb_j)_{j\in \mathcal J}$. 
Set $\xi := \dfrac{\log n}{50 \log(r^2/\log n)}$. By a direct computation we get for every $r\ge r_0$ that $\xi\ge 2\lambda$ and that 
\begin{align*}
    \mathbb P\left(\forall j\in \mathcal J, |B_j|\le \xi\right) 
    &=\hspace{0.3em} \prod_{j\in \mathcal J} \mathbb P\left(|B_j|\le \xi\right)\\ 
    &\le\hspace{0.3em} \prod_{j\in \mathcal J} 2\mathbb P\left(|B_j|= \xi\right)\\ 
    &=\hspace{0.3em} \prod_{j\in \mathcal J} 2\exp\left(-\lambda\right) \dfrac{\lambda^{\xi}}{\xi !}\\ 
    &\le\hspace{0.3em} 2^{n^{0.2}}\exp\left(-n^{0.2}\lambda\right) \left(\dfrac{\lambda^{\xi}}{\xi !}\right)^{n^{0.2}}\\
    &\le\hspace{0.3em} 2^{n^{0.2}}\exp\left(-n^{0.2}\lambda\right) \left(\dfrac{e \lambda}{\xi}\right)^{n^{0.2}\xi}\\
    &=\hspace{0.3em} 2^{n^{0.2}}\exp\left(- n^{0.2} \dfrac{\pi \log^2 n}{(16 r)^2}\right) \left(\dfrac{e \pi \log^2 n/(16 r)^2}{\log n/(50 \log(r^2/\log n))}\right)^{\frac{n^{0.2}\log n}{50\log(r^2/\log n)}}\\
    &=\hspace{0.3em} 2^{n^{0.2}}\exp\left(- n^{0.2} \dfrac{\pi \log^2 n}{256 r^2}\right) \left(\dfrac{50 e \pi \log(r^2/\log n)}{256 r^2/\log n}\right)^{\frac{n^{0.2}\log n}{50\log(r^2/\log n)}}\\
    &=\hspace{0.3em} \left(2\exp\left(\dfrac{\log n}{50\log(r^2/\log n)} \log\left(\dfrac{50 e \pi \log(r^2/\log n)}{256 r^2/\log n}\right) - \dfrac{\pi \log^2 n}{256 r^2}\right)\right)^{n^{0.2}}.
\end{align*}
We conclude that the last probability is $o(1/\sqrt{n})$ since 
\begin{equation*}
    \dfrac{\log n}{50\log(r^2/\log n)} \log\left(\dfrac{50 e \pi \log(r^2/\log n)}{256 r^2/\log n}\right) < 0 \text{ and } \dfrac{\pi \log^2 n}{256 r^2}\ge 1
\end{equation*}
for every $r = o(\log n)$ and $r \ge r_0$, which gives an upper bound of $(2/e)^{n^{0.2}} =o(1/\sqrt{n})$.

By de-Poissonization we deduce that a.a.s.\ there is a ball $\Bb_i$ among $(\Bb_j)_{j\in \mathcal J}$ containing at least $\xi$ vertices of the random geometric graph $G\in \mathcal T(n, r)$. In particular, since by definition of the set $\mathcal J$ the vertices $\Bb_i\cap V(G)$ cannot be distinguished by a sensor outside $\Bb_i$, the robber can always escape from the cops in the presence of only $\xi - 2 = \Omega(\log n/\log(r^2/\log n))$ sensors by choosing to remain in the ball $\Bb_i$ after each step, thereby finishing the proof of the lower bound of Part 4) in Theorem~\ref{thm:main}.
\end{proof}

\section{Outlook and open problems}

In this paper we determined up to a multiplicative poly-logarithmic factor the localization number of the random geometric graph. As already mentioned in the introduction, for any graph $G$, we have $\zeta(G) \le \beta(G)$. Whereas in $\mathcal{G}(n,p)$ these two parameters are relatively close to each other, for many values of $r$ this is not the case for $G \in \RT$, as the following lemma shows. In fact, in view of the lower bound on $\zeta(G)$ given by Part 1) in Theorem~\ref{thm:main}, the following lemma shows that for $r \ll n^{3/10}$ the bounds are far from each other. For the sake of completeness, we also show the upper bound our approach for the localization number gives for $\beta(G)$. 
\begin{lemma}
Let $G \in \RT$. A.a.s.\ we have
\begin{enumerate}
    \item[(i)] If $r \gg 1$ and $r \le c\sqrt{n/\log n}$ for small enough $c > 0$,  then $\beta(G) = \Omega(n/r^2)$.
    \item[(ii)] If $\log^{3/2}n \le r \le \sqrt{n}/4$, then $\beta(G)=O((n \log^{2/3} n)/r^{2/3})$.
\end{enumerate}

\end{lemma}
\begin{proof}
We prove part~(i) in a Poissonized setup, de-Poissonizing only at the end. Tessellate the torus into square cells of width $3r$, and consider in each cell $C$ the inner cell $c$ of width $0.1r$ centered at the same point as $C$. Subdivide further $c$ into subcells of width $1/r$. Consider the event $\mathcal{E}_C$ that inside the inner cell $c$ of $C$ there is a subcell having exactly two vertices $u, v$, and that there is no vertex at distance at most $r$ from $u$ ($v$, respectively), while at the same time being at distance more than $r$ from $v$ ($u$, respectively). Observe that if $\mathcal{E}_C$ holds, then either $u$ or $v$ has to be taken into a minimum set of sensors which guarantees that the cops can win in one round. Moreover, for different cells $C, C'$ the corresponding events $\mathcal{E}_C, \mathcal{E}_{C'}$ are independent. Denote by $X_C$ the indicator random variable for $\mathcal{E}_C$. 

The probability that no subcell has exactly $2$ vertices therein is equal to
$$
\left( 1 - \frac {(1/r^2)^2}{2} e^{-1/r^2} \right)^{(0.01+o(1)) r^4} = \left( 1 - \frac {0.5+o(1)} {r^4} \right)^{(0.01+o(1)) r^4} = e^{-0.005} + o(1). 
$$
Condition on the event that there is a subcell having exactly two vertices $u, v$, and observe that $\Bb(u,r) \Delta \Bb(v,r)$ is completely contained in $C \setminus c$. Thus, since by Observation~\ref{ob AB}, $|\Bb(u,r) \Delta \Bb(v,r)| \le 4 \cdot \sqrt{2}(1+o(1))$, we have 
$$
\Prob (X_C=1) \ge (1 - e^{-0.005}+o(1)) \cdot (e^{-4\sqrt{2}} + o(1)) \ge 10^{-5}.  
$$
Denote $X = \sum_C X_C$. Observing that there are $\Theta(n/r^2)$ cells implies that $\mathbb{E}(X)=\Theta(n/r^2)$. 
Hence, since $r \le c\sqrt{n/\log n}$ for a small enough constant $c > 0$, part~(i) follows by Chernoff's bound~\eqref{chern} together with the de-Poissonization argument given in Section~\ref{sec:Poisson}. 

For part~(ii), recall by Lemma~\ref{lem:vertices_in_crown} that a.a.s.\ for any pair of points with positions $A$, $B$ such that $d_T(A,B) \ge \varepsilon := (\log n/r)^{1/3}$, the number of vertices in $\Bb(A,r) \Delta \Bb(B,r)$ is at least $\min(\eps, 2r)r$. Now, for every vertex, put a sensor on it independently of all others, with probability $C\log^{2/3} n/r^{2/3}$ for some large enough constant $C > 3$ (thus constructing a random set of expected size $C n \log^{2/3}n /r^{2/3}$), and then add for any pair of vertices that is not distinguished yet one of the two vertices. We now show that the number of vertices added is at most of the same order, thus proving the desired upper bound of $O(n \log^{2/3}n /r^{2/3})$. To do so observe that by considering the family $\mathcal F$ of squares (defined right before the statement of Lemma~\ref{lem:second_phase}) every pair of vertices at distance at most $0.1r$ is inside one square $S \in \mathcal F$. Hence, for $r \ge \log^{3/2} n$, by Lemma~\ref{lem:pairs_at_given_distance}~(b) (note that $(\log n/r)^{1/3} \le r^{-0.1}$ for the range of $r$) together with a union bound over all squares $S$ in $\mathcal F$, the number of pairs at distance at most $(\log n/r)^{1/3}$ is at most $|\mathcal F| \cdot (2+o(1)) \cdot 10^7 r^{4/3}\log^{2/3}n=O(n \log^{2/3} n/r^{2/3})$, and we may add for each such pair one vertex. For all other pairs of vertices at distance at least $(\log n/r)^{1/3}$, by a union bound, the probability that there exists a pair not distinguished by the random set is at most
$$
n^2 \left(1 - \frac{C\log^{2/3} n}{r^{2/3}}\right)^{r^{2/3}\log^{1/3}n} =o(1/n),
$$
and hence a.a.s. all such vertices will be distinguished by the random set.
\end{proof}
Together with the already obtained lower bounds on $\zeta(G)$ and using the fact that $\zeta(G) \le \beta(G)$, we thus have the following bounds on the metric dimension:
\begin{theorem}\label{thm:metric_dimension}
Let $G \in \RT$. A.a.s.\ the following bounds hold: 
\begin{itemize}
\item If $1 \ll r \le \sqrt{n}/4$, then
$\Omega\left( \max (n/r^2,r^{4/3}/\log^{1/3} n)\right) = \beta(G)$.
\item If $\log^{3/2}n \le r \le \sqrt{n}/4$, then $\beta(G)= O\left(n \log^{2/3}n /r^{2/3}\right)$.
\end{itemize}
\end{theorem}

We finish the paper with the following natural open questions.
\begin{openproblem}
Let $G \in \RG$. Theorem~\ref{thm:metric_dimension} implies that a.a.s.\ $\beta(G) = n^{2/3+o(1)}$, provided that $r = n^{1/2+o(1)}$ (and $r \le \sqrt{n}/4$) but the bounds are far away from each other for sparser graphs.  What is the value of $\beta(G)$ for $G \in \RG$?
\end{openproblem}

\begin{openproblem}
Let $G \in \RG$. Our results imply that a.a.s.\ $\zeta(G) / \beta(G) = o(1)$, provided that $r \ll n^{3/10}$. What about denser graphs?
\end{openproblem}

\begin{openproblem}
Let $G \in \RG$. Our results give relatively tight bounds for the localization number of $G$, provided that $r \ge \log^{3/2} n$. The bounds for sparser graphs are slightly worse. For example, our lower bound in the range of $r \in [r_0, \log n]$ is not monotonic but there is no apparent reason why it should not be monotonic. Moreover, what is the localization number close to the threshold of connectivity? 
\end{openproblem}

\paragraph{Acknowledgements.} The authors are grateful to the two anonymous referees for a number of useful comments.

\bibliographystyle{plain}
\bibliography{References}

\end{document}